\documentclass{elsarticle} 
\journal{Computer Methods in Applied Mechanics and Engineering}

\usepackage[]{fullpage}

\usepackage{diagbox,amssymb, amsthm, amsmath,epsfig,graphics,psfrag,graphicx,color,fancyhdr,subfigure,arydshln,comment}
\usepackage{multirow}

\definecolor{darkred}{rgb}{0.82,0.15,0.20}
\definecolor{darkblue}{rgb}{0.82,0.15,0.12}
\definecolor{lightblue}{rgb}{0.22,0.45,0.70}
\definecolor{lightgreen}{rgb}{0.1,0.6,0.10}

\usepackage[colorlinks=true,breaklinks=true]{hyperref}
\hypersetup{
	colorlinks=true,
}

\numberwithin{equation}{section}
\numberwithin{figure}{section}
\numberwithin{table}{section}
\newcommand\bd{\boldsymbol{d}}
\newcommand\bu{\boldsymbol{u}}
\newcommand\bv{\boldsymbol{v}}
\newcommand\bw{\boldsymbol{w}}
\newcommand\bx{\boldsymbol{x}}

\newcommand\bz{\boldsymbol{z}}
\newcommand\nn{\boldsymbol{n}}
\renewcommand\gg{\boldsymbol{g}}
\newcommand\ff{\boldsymbol{f}}
\newcommand\bt{\boldsymbol{t}}

\newcommand\bsigma{\boldsymbol{\sigma}}
\newcommand\beps{\boldsymbol{\epsilon}}

\newcommand\bI{\mathbf{I}}

\newcommand\bH{\mathbf{H}}
\newcommand\bL{\mathbf{L}}
\newcommand\bV{\mathbf{V}}
\newcommand\bW{\mathbf{W}}

\newcommand\cT{\mathcal{T}}
\newcommand\cA{\mathcal{A}}

\newcommand\cE{\mathcal{E}}

\newcommand\cR{\mathcal{R}}

\newcommand\cRC{\cR_C}
\newcommand\cRD{\cR_D}
\newcommand\cRF{\cR}

\newcommand\bdiv{\mathop{\mathbf{div}}\nolimits}
\newcommand\vdiv{\mathop{\mathrm{div}}\nolimits}

\newcommand\cero{\boldsymbol{0}}
\newcommand{\jump}[1]{\ensuremath{[\![#1]\!]} }

\newtheorem{theorem}{Theorem}[section]
\newtheorem{lemma}{Lemma}[section]

\theoremstyle{remark}

\theoremstyle{example}
\newtheorem{example}{Example}[section]

\allowdisplaybreaks

\begin{document}
\hypersetup{
		linkcolor=lightblue,
		urlcolor=darkred,
		citecolor=lightgreen
	}
\begin{frontmatter}

\title{Parameter-robust methods for the Biot-Stokes interfacial coupling without Lagrange multipliers}

\author[kth]{Wietse M. Boon\fnref{wmb}}
\ead{wietse@kth.se}
\author[uio]{Martin Hornkj{\o}l}
\ead{marhorn@math.uio.no}
\author[simula]{Miroslav Kuchta\corref{cor1}\fnref{mk}}
\ead{miroslav@simula.no}
\author[uio,simula]{Kent-Andr\'e Mardal\fnref{kam}}
\ead{kent-and@math.uio.no}
\author[monash,sechenov]{Ricardo Ruiz-Baier\fnref{rrb}}
\ead{ricardo.ruizbaier@monash.edu}

\address[kth]{KTH Royal Institute of Technology, Lindstedtsv\"agen 25, 114 28, Stockholm, Sweden}
\address[uio]{Department of Mathematics,  University of Oslo, Norway}
\address[simula]{Simula Research Laboratory, 0164 Oslo, Norway}
\address[monash]{School of Mathematics, 
Monash University, 9 Rainforest Walk, Melbourne 3800 VIC,  Australia} 
\address[sechenov]{Institute of Computer Science and Mathematical Modelling, Sechenov University, Moscow, Russian Federation} 

\cortext[cor1]{Corresponding author}

\fntext[wmb]{WMB was supported by the Dahlquist Research Fellowship.}
\fntext[mk]{MK acknowledges support from the Research Council of Norway (NFR) grant 303362.}
\fntext[kam]{KAM acknowledges support from the Research Council of Norway, grant 300305 and 
	301013.}
\fntext[rrb]{ RRB acknowledges support from the Monash Mathematics Research Fund S05802-3951284
        and from the Ministry of Science and Higher Education of the Russian Federation within the
        framework of state support for the creation and development of World-Class Research Centers
        ``Digital biodesign and personalised healthcare" No. 075-15-2020-926.}

\begin{abstract}
In this paper we advance the analysis of discretizations for a fluid-structure interaction model of the monolithic coupling between the free flow of a viscous Newtonian fluid and a deformable porous medium separated by an interface. A five-field mixed-primal finite element scheme is proposed solving for Stokes velocity-pressure and Biot displacement-total pressure-fluid pressure. Adequate inf-sup conditions are derived,  and one of the distinctive features of the formulation is that its stability is  established robustly in all material parameters. 
We propose  robust preconditioners for this perturbed saddle-point problem using appropriately weighted operators in fractional Sobolev and metric spaces at the interface. The performance is  corroborated by several test cases, including the application to  interfacial flow in the brain. 
\end{abstract}

\begin{keyword} Transmission problem\sep Biot-Stokes coupling\sep total pressure\sep mixed finite elements\sep operator preconditioning\sep brain poromechanics. 
\MSC{
 65N60, 76S05, 74F10, 92C35.
  }
\end{keyword}

\end{frontmatter}

\section{Introduction}\label{sec:intro}
\subsection{Scope} 
We address the construction of appropriate monolithic solvers for multiphysics fluid-poromechanical couplings interacting through 
an interface. Particular attention is payed to tracking parameter dependence of the continuous and discrete formulations so that the resulting numerical methods are robust 
with respect to typical scales in material constants spanning over many orders of magnitude. We adopt a multi-domain approach, where appropriate conditions for the 
coupling through the shared interface need to be imposed. We use the conditions proposed in \cite{murad01} (although, other
forms and dedicated phenomena could be incorporated without much effort, such as fluid entry resistance
\cite{showalter05,bergkamp20}). The recent literature contains various numerical methods for (Navier-)Stokes/Biot
interface formulations including mixed, double mixed, monolithic, segregated, conforming, non-conforming, and DG discretizations 
\cite{ager19,ambar18,badia09,bukac15,caucao21,cesme17,cesme20,li20,taffetani20,wen20,wen21,wilfrid20}. 

In \cite{taffetani20,ruiz21} (and starting from the Biot-Stokes equations 
advanced in \cite{ambar18,ambar20}) the authors rewrite the poroelasticity equations using displacement, fluid pressure and total pressure (also as in the poromechanics formulations from  \cite{buerger21,kumar20,oyarzua16}).  
Since fluid pressure in the poroelastic domain 
has sufficient regularity, no Lagrange multipliers are needed to enforce the coupling conditions, which  resembles the different formulations for Stokes-Darcy advanced in \cite{boon21SD,chidyagwai2016constraint,holter20,karper09}. Another advantage 
of the three-field Biot formulation is its robustness with respect to the Lam\'e constants 
of the poroelastic structure. This robustness is of particular importance when we test the flow response to changes in the 
material properties of the skeleton and when the solid is nearly incompressible. 
The work \cite{taffetani20} focuses on the stability analysis and its precise implications on the asymptotics of the interface conditions  when the permeability depends on porosity heterogeneity, whereas \cite{ruiz21} addresses the stability of the semi- and fully discrete problems, and the application to interfacial flow in the eye.  Here, we extend these works by concentrating on deriving robust stability, on designing efficient block preconditioners (robust with respect to all material parameters) following the general theoretical formalism from \cite{mardal-winther}, and on the simulation of free flow interacting with interstitial flow in the brain. In such a context (and in the wider class of problems we consider in this paper), tissue permeability is of the order of 10$^{-15}$m$^2$, and the incorporation of tangential interface transmission conditions usually involves terms that scale inversely proportional to the square root of 
permeability. Moreover, the solid is nearly incompressible, making the first Lam\'e parameter significantly larger than 
the other mechanical parameters and exhibiting volumetric locking for some types of displacement-based formulations.  Other flow regimes that are challenging include low-storage cases \cite{mardal20}. It is then important that the stability and convergence of the numerical approximations are preserved within the parameter ranges of interest. 

Here we follow \cite{lee17,lee19,oyarzua16, boon20} and use parameter-weighted norms to achieve robustness. However, as
we will see, combining proper preconditioners for Stokes and Biot single-physics problems is not sufficient  
for the interface coupled problem. In fact, the condition number of the preconditioned system, although robust in mesh size, grows     
like the square root of the ratio between fluid viscosity and permeability. 
This phenomenon is demonstrated in Example~\ref{ex:naive}, below. That is, the efficiency of seemingly
natural preconditioners varies with the material parameters.  
In order to  regain stability with respect to all parameters, we include both an additional fractional term involving the pressure and a metric term coupling the tangential fluid velocity and solid displacement at the interface, hereby increasing the  regularity at the interface in a proper parameter dependent manner.  This strategy draws inspiration from similar approaches employed in the design of robust solvers for Darcy and Stokes-Darcy couplings \cite{baerland20,boon21SD,holter20}. 

\subsection{Outline}
We have organised the contents of this paper in the following manner. The remainder of this section contains preliminaries on notation and functional spaces to be used throughout the manuscript. Section~\ref{sec:model} outlines the main details of the   balance equations, stating typical interfacial and boundary conditions, and restricting the discussion to the steady Biot-Stokes coupled problem. There we also include the weak formulation and  demonstrate that simple diagonal preconditioners based on standard norms do not lead to robustness over the whole parameter range. This issue is addressed in Section~\ref{sec:analysis} where we show well-posedness of the system  using a global inf-sup argument with  parameter weighted operators  in fractional spaces, which in turn assist in the design of robust solvers  by operator preconditioning.  Section~\ref{sec:results} discusses finite element discretization of the coupled problem using both conforming and non-conforming elements; and it also contains numerical experiments demonstrating robustness of the fractional preconditioner and its feasibility for simplified simulations of   interfacial flow in the brain.

\subsection{Preliminaries}
Let us consider a spatial domain $\Omega\subset \mathbb{R}^d$, where $d=2,3$, disjointly split into $\Omega_F$ and $\Omega_P$. These subdomains respectively represent the region
filled with an incompressible fluid and the elastic porous medium (a deformable solid matrix or an array of solid particles). We will denote by $\nn$ the unit normal 
vector on the boundary $\partial\Omega$, and by $\Sigma=\partial\Omega_F\cap \partial\Omega_P$ 
the interface between the two subdomains, which is assumed sufficiently regular. We 
adopt the convention that 
on $\Sigma$ the normal vector points from $\Omega_F$ to $\Omega_P$. 
We also define the boundaries $\Gamma_F = \partial\Omega_F\setminus\Sigma$ and 
$\Gamma_P = \partial\Omega_P \setminus\Sigma$. The sub-boundary $\Gamma_F$ is further decomposed between $\Gamma_F^{\bu}$ and $\Gamma_F^{\bsigma}$ where we impose, respectively,  no slip velocities and zero normal total stresses. Similarly, we split 
$\Gamma_P$ into $\Gamma_P^{p_P}$ and $\Gamma_P^{\bd}$ where we prescribe zero traction and clamped boundaries, 
respectively. 
For the analysis, the setup of trace spaces needs that $\mathrm{dist}(\Sigma,\Gamma_P^{p_P}) >0$ and that $\mathrm{dist}(\Sigma,\Gamma_F^{\bsigma}) >0$, which can be satisfied if the interface meets the boundary at the Biot displacement and
Stokes velocity boundaries (see Figure~\ref{fig:sketch}, left), 
where $\Gamma_P^{\bd} = \Gamma_{1,P}^{\bd} \cup \Gamma_{2,P}^{\bd}$ and  $\Gamma_F^{\bu} = \Gamma_{1,F}^{\bu} \cup \Gamma_{2,F}^{\bu}$. 
Our numerical tests will also include cases where the interface intersects boundaries $\Gamma_F^{\bsigma}$ and $\Gamma_P^{p_P}$ on
the Stokes and Biot sides, respectively. 

\begin{figure}[t!]
\centering
   \includegraphics[width=0.42\textwidth]{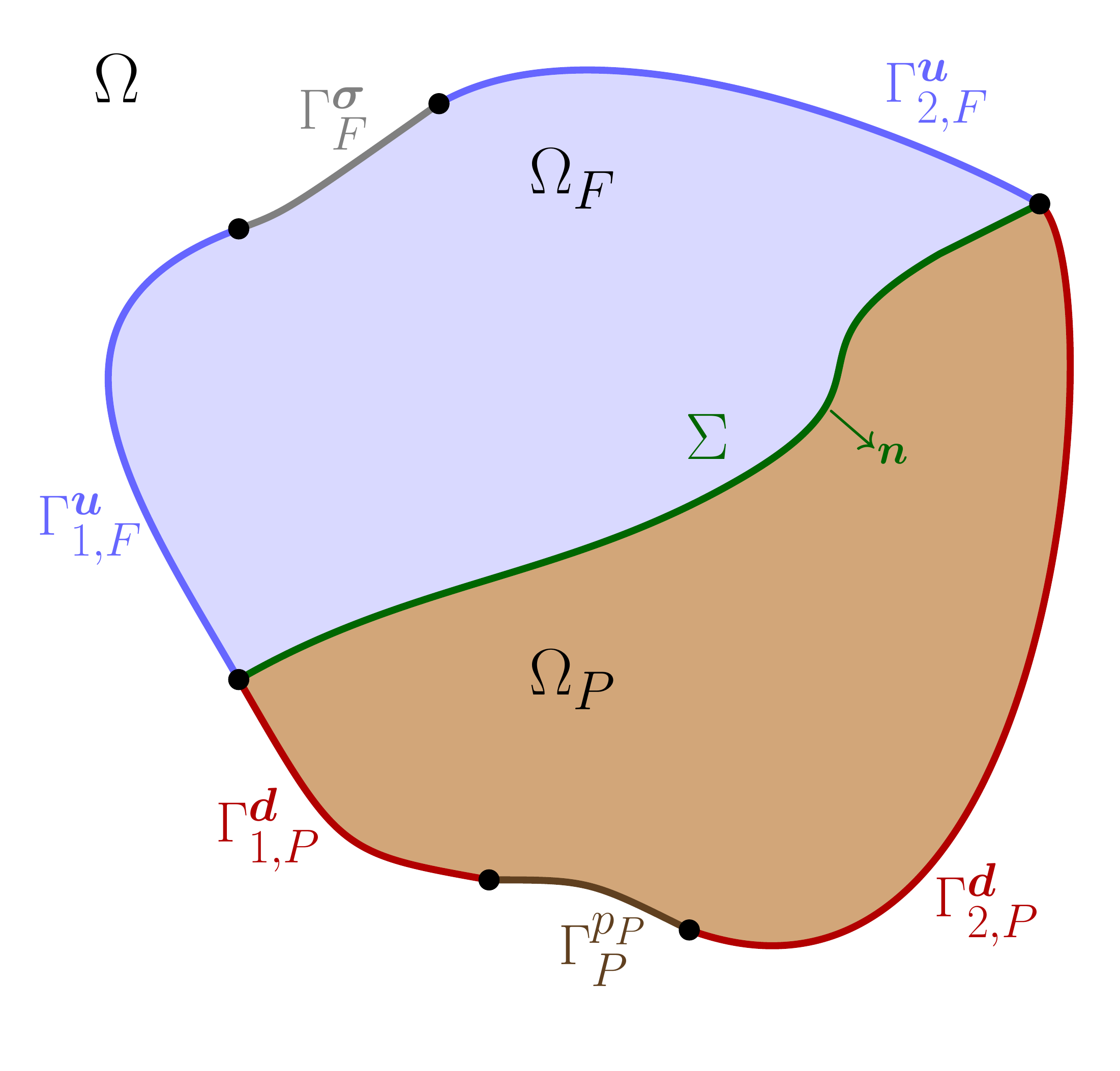}  
   \qquad  \includegraphics[width=0.42\textwidth]{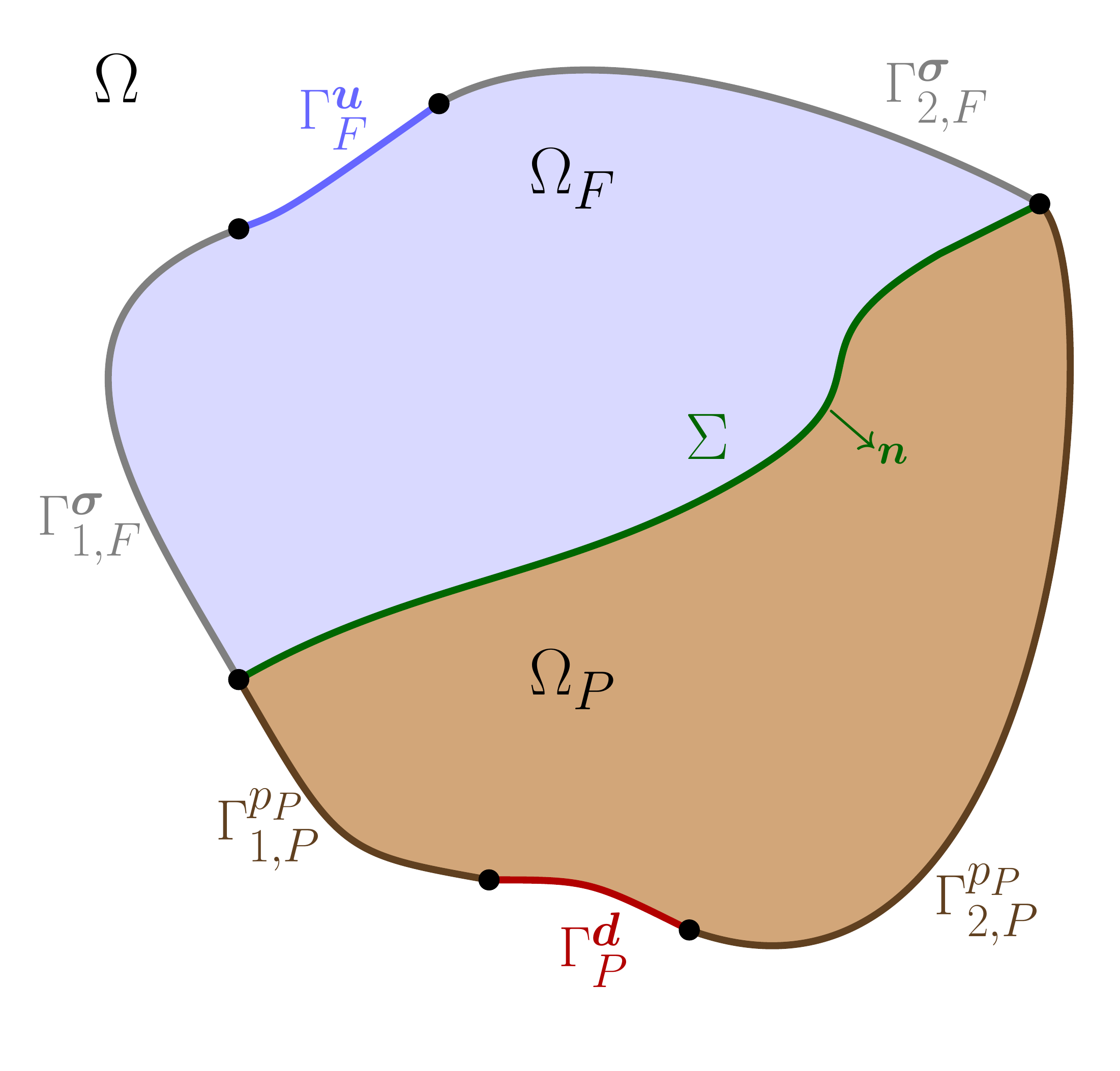}
  \vspace{-10pt}
  \caption{A configuration of subdomains and boundary partition for the
      Biot-Stokes coupled problem. (Left) Setup assumed for the analysis in Theorem \ref{theo:iso}.
      (Right) Example of another configuration investigated by some of the numerical experiments in Section \ref{sec:num_robust_neumann}.}  \label{fig:sketch}
\end{figure}

For generic Sobolev spaces $X,Y$ and a scalar $\zeta>0$, the weighted space $\zeta X$ refers to $X$ endowed with 
the norm $\zeta \|\cdot\|_X$. The intersection $X\cap Y$ provided with the norm $\| v \|^2_{X\cap Y} = \|v\|^2_X + \|v\|_Y^2$,  is a Hilbert space \cite{baerland20, bergh2012interpolation}. 

Vector fields and vector-valued spaces will be written in boldface. In addition, by $L^2(\Omega)$ we will denote 
the usual Lebesgue space of square integrable functions and $H^m(\Omega)$ denotes the usual Sobolev space with weak derivatives of order up to $m$ in $L^2(\Omega)$, and $\bH^m(\Omega)$ denotes its vector counterpart. In addition, 
$\bL_{\bt}^2(\Sigma)$ will denote the space of functions $\bz:\Omega_i\to \mathbb{R}^d$, $i\in \{F,P\}$, {such that $\bz-\nn(\bz\cdot \nn)\in \bL^2(\Sigma)$.}

An $L^2(\Omega)$ (as well as $\bL^2(\Omega)$) inner product over a generic bounded domain $\Omega$ is
denoted as $(\cdot, \cdot)_{\Omega}$. The symbol $\langle\cdot,\cdot\rangle_\Sigma$ will denote the pairing between the trace functional space $H^{1/2}(\Sigma)$ and its dual $H^{-1/2}(\Sigma)$, and   we will also write $\langle \cdot,\cdot\rangle$ to denote other, more general, duality pairings. Moreover, for $\bz \in \bH^1(\Omega_i) \cap \bL^2_{\bt}(\Sigma)$, 
its normal and tangential traces defined by bounded surjective maps, will be denoted by $T_{\nn}\bz \in H^{1/2}(\Sigma)$ and 
{$T_{\bt}\bz \in \bL^2(\Sigma)$}, respectively  
\cite{holter20}. We also remark that, while $H^{1/2}(\Sigma)$ is a subspace of $L^2(\Sigma)$ as a set, 
the tangential trace of $\bz \in \zeta_1 \bH^1(\Omega_i) \cap \zeta_2 \bL^2_{\bt}(\Sigma)$
is in $\zeta_1 \bH^{1/2}(\Sigma) \cap \zeta_2 \bL^2(\Sigma)$ 
and that in our setting it is important to track the parameter dependence for the sake of  robustness. 

Pertaining to the poroelastic domain, and assuming momentarily that $|\Gamma^{\bd}_P|=0$, we recall from \cite[Section 2.4.2]{Gatica2014} the definition of the space 
 \begin{equation}
\label{eq:H00} 
 H_{00}^{1/2}(\Sigma) = \{\eta \in H^{1/2}(\Sigma): \ E_{00}(\eta)\in H^{1/2}(\partial\Omega_P)\},\end{equation}
supplied with the norm 
\[ \|\eta\|_{1/2,00,\Sigma} := \| E_{00}(\eta)\|_{1/2,\partial\Omega_P},\]
where 
$E_{00}:H^{1/2}(\Sigma)\to H^{1/2}(\partial\Omega_P)$ denotes the extension-by-zero operator 
\[E_{00}(\eta)= \begin{cases}\eta & \text{on $\Sigma$},\\ 0 & \text{on $\partial\Omega_P\setminus\Sigma$},\end{cases} \qquad \forall \eta \in H^{1/2}(\Sigma).\]

Furthermore, the restriction of $\psi\in H^{-1/2}(\partial\Omega_P)$ to $\Sigma$, defined as 
\[ \langle \psi|_{\Sigma},\eta \rangle_{\Sigma} = \langle\psi,E_{00}(\eta)\rangle_{\partial\Omega_P} \quad \forall \eta \in H^{1/2}_{00}(\Sigma),\]
belongs to the dual $[H_{00}^{1/2}(\Sigma)]'$  of $H_{00}^{1/2}(\Sigma)$. 
Its norm is 
\begin{equation}\label{eq:norm-12}
\|\psi|_{\Sigma}\|_{-1/2,00,\Sigma} := \sup_{0\neq \eta\in H_{00}^{1/2}(\Sigma)}\frac{\langle \psi|_{\Sigma},\eta\rangle_{\Sigma} }{\|\eta\|_{1/2,00,\Sigma}} = \sup_{0\neq \eta\in H_{00}^{1/2}(\Sigma)}\frac{\langle \psi,E_{00}(\eta)\rangle_{\partial\Omega_P} }{\|E_{00}(\eta)\|_{1/2,\partial\Omega_P}}.\end{equation}

The boundary setup in Figure~\ref{fig:sketch} is such that $|\Gamma^{\bd}_P|\cdot|\Gamma^{p_P}_P|>0$, and therefore a modification of \eqref{eq:H00}-\eqref{eq:norm-12}   
is required. 
With that purpose, we note that any $\eta \in H^{1/2}(\Sigma)$ can be continuously extended to $\varphi$ in the space 
\begin{equation}\label{eq:H0p}
H_{01}^{1/2}(\Sigma) := \{ \eta \in H^{1/2}(\Sigma):
E_{01}(\eta) \in H^{1/2}(\Gamma_1) \},\end{equation}
where $\bar \Gamma_1 := \bar \Sigma \cup \bar{\Gamma}_{1,P}^{\bd}\cup \bar{\Gamma}_{2,P}^{\bd}$ (see Figure~\ref{fig:sketch}, left), and 
\[E_{01}(\eta)= \begin{cases}\eta & \text{on $\Sigma$},\\ 0 & \text{on $\Gamma_P^{\bd}$},\end{cases} \qquad \forall \eta \in H^{1/2}(\Sigma).\]
This treatment can be interpreted as replacing $ H_{00}^{1/2}(\Sigma) $ by the space of functions  $\eta \in H^{1/2}(\Sigma)$ such that  $\xi^{-1/2} \eta \in L^{2}(\Sigma)$, where $\xi$ is a sufficiently regular trace function, positive on $\Sigma$, and vanishing only on $\Gamma^{\bd}_P$ (see, for example,  \cite{leroux93}). 

The norm of the resulting extension $\varphi\in H_{01}^{1/2}(\Sigma)$ is defined analogously as before, 
\begin{equation}\label{eq:norm-H0p}
 \|\varphi\|_{1/2,01,\Sigma} := \| E_{01}(\varphi) \|_{1/2,\Gamma_1},\end{equation}
and therefore   the restriction of a distribution to the interface, $\psi|_\Sigma$, is in the dual space $[H_{01}^{1/2}(\Sigma)]'$ and its norm is 
\begin{equation}\label{eq:norm-H0p-dual}
\|\psi|_{\Sigma}\|_{-1/2,01,\Sigma} := \sup_{0\neq \eta\in H_{01}^{1/2}(\Sigma)}\frac{\langle \psi|_{\Sigma},\eta\rangle_{\Sigma} }{\|\eta\|_{1/2,01,\Sigma}} = \sup_{0\neq \eta\in H_{01}^{1/2}(\Sigma)}\frac{\langle \psi,E_{01}(\eta)\rangle_{\Gamma_1} }{\|E_{01}(\eta)\|_{1/2,\Gamma_1}}.\end{equation}

\section{Governing equations and weak formulation}\label{sec:model}

The momentum and mass  balance equations for the flow 
in the fluid cavity are given by Stokes equations written in terms 
of fluid velocity $\bu$ and fluid pressure $p_F$, whereas 
the non-viscous filtration flow through the porous skeleton can be described by Darcy's law 
in terms of pressure head $p_P$, and the porous matrix elastostatics are 
 stated in terms of the solid displacement $\bd$. The coupled Biot-Stokes equations 
 arising after a backward Euler semi-discretization in time, with time step $\Delta t$, read 
\begin{subequations}\label{problem}
 \begin{align}
\label{eq:momentumA}
 -\bdiv[2\mu_f \beps(\bu) - p_F\bI] & = \rho_f \gg & \text{in $\Omega_F$},\\ 
\label{eq:massA}
\vdiv \bu & = 0 & \text{in $\Omega_F$},\\
\label{eq:momentumM}
-\bdiv[2\mu_s\beps(\bd) -\varphi \bI ]& = \rho_s\ff & \text{in $\Omega_P$},\\
\label{eq:phi}
\varphi - \alpha p_P + \lambda \vdiv\bd & = 0 & \text{in $\Omega_P$},\\
\label{eq:massM}
\bigl(C_0 + \frac{\alpha^2}{\lambda}\bigr)\frac{1}{\Delta t} p_P - \frac{\alpha}{(\Delta t) \lambda} \varphi  - \vdiv \biggl(\frac{\kappa}{\mu_f} \nabla p_P - \rho_f \gg\biggr) & = m_P& \text{in $\Omega_P$},
\end{align}\end{subequations}
where $\ff$ is a vector field of body loads, $\gg$ is the gravity acceleration, $\mu_f$ is the fluid viscosity, 
$\beps(\bu) = \frac{1}{2}(\nabla\bu+\nabla\bu^{\tt t})$ 
is the strain rate tensor, and $\beps(\bd) = \frac{1}{2}(\nabla\bd+\nabla\bd^{\tt t})$ is the 
infinitesimal strain tensor, $\rho_f,\rho_s$ are the density of the 
fluid and solid, respectively, $\lambda,\mu_s$ are the first and second Lam\'e constants of the solid, $\kappa$ is the heterogeneous tensor of permeabilities (satisfying $|\bw|^2 \lesssim \bw\cdot \kappa(\bx) \bw$ a.e. in $\Omega_P$ and for all $\bw\in \mathbb{R}^d$);  $m_P$ is a source/sink term for the fluid pressure (which also includes pressures in the previous backward Euler time step);  and $C_0, \alpha$ are the 
total storage capacity and Biot-Willis poroelastic coefficient. Here we have used 
the total pressure $\varphi: = \alpha p_P - \lambda\vdiv\bd$, as an additional unknown \cite{lee17,oyarzua16}. 

We furthermore supply boundary conditions as follows 
\begin{subequations}\label{eq:bc}
\begin{align}\label{eq:bcAD}
\bu & = \cero & \text{on $\Gamma_F^{\bu}$},\\
\label{eq:bcAN}
[2\mu_f\beps(\bu)-p_F\bI]\nn & = \cero& \text{on $\Gamma_F^{\bsigma}$},\\
\label{eq:bcMd}
\bd = \cero\quad \text{and} \quad \frac{\kappa}{\mu_f} \nabla p_P \cdot\nn &= 0 &\text{on $\Gamma_P^{\bd}$},\\
\label{eq:bcMu}
[2\mu_s\beps(\bd) -\varphi \bI]\nn = \cero \quad\text{and}\quad p_P&=p_{0}  &\text{on $\Gamma_P^{p_P}$}.
\end{align}\end{subequations}
In order to close the system, we consider the classical transmission conditions on $\Sigma$ accounting for the continuity of normal 
fluxes, momentum balance, equilibrium of fluid normal stresses, and the so-called Beavers-Joseph-Saffman condition for tangential fluid forces \cite{bukac15,showalter05}, which in the present setting reduce to 
\begin{subequations}
\begin{align}
\label{eq:inter-U}
\bu\cdot\nn & = (\frac{1}{\Delta t} \bd  - \frac{\kappa}{\mu_f} \nabla p_P )\cdot \nn & \text{on $\Sigma$},\\
\label{eq:inter-sigma}
(2\mu_f \beps(\bu) - p_F\bI)\nn & = (2\mu_s\beps(\bd) -\varphi \bI ) \nn & \text{on $\Sigma$},\\
\label{eq:inter-sigmaf}
-\nn\cdot (2\mu_f \beps(\bu) - p_F\bI)\nn & = p_P & \text{on $\Sigma$},\\
\label{eq:inter-bjs}
- \nn\cdot (2\mu_f \beps(\bu) - p_F\bI)\bt_j & = \frac{\gamma\mu_f}{\sqrt{\kappa}} (\bu - \frac{1}{\Delta t} \bd)\cdot \bt_j, \quad 1\leq j\leq d-1 & \text{on $\Sigma$},
\end{align}\end{subequations}
where $\gamma>0$ is the slip rate coefficient depending on the geometry of the domain, and we recall that the normal $\nn$ on the interface is understood as pointing from the fluid domain $\Omega_F$ towards the porous structure $\Omega_P$, while $\bt_j$, $1\leq j\leq d-1$  are orthonormal tangent vectors on $\Sigma$, normal to $\nn$.

We proceed to 
test \eqref{eq:momentumA}-\eqref{eq:massM} against suitable smooth functions and to integrate over the corresponding subdomain. The challenging model parameters  are $\mu_f$, $C_0$, $\lambda$, $\gamma$, $\alpha$, and the magnitude of $\kappa$. Therefore, we will concentrate on the specific case where $\Delta t = 1$, and the model parameters (in particular $\kappa$) are spatially constant.
 Following \cite{karper09}, after applying 
integration by parts wherever adequate and using  the transmission conditions \eqref{eq:inter-U}-\eqref{eq:inter-bjs}, we arrive at 
the following remainder on the interface 
\[
\langle p_P, (\bv - \bw)\cdot \nn \rangle_\Sigma +  \frac{\gamma\mu_f}{\sqrt{\kappa}} \sum_{j=1}^{d-1}\langle (\bu - \bd)\cdot \bt_j,  (\bv-\bw)\cdot \bt_j\rangle_\Sigma +\langle (\bu-\bd)\cdot \nn, q_P \rangle_\Sigma,\]
which is well-defined and therefore no additional Lagrange multipliers are required to realize the coupling conditions. Also, in view of the character of the resulting variational forms in combination with the specification of boundary conditions \eqref{eq:bc}, we  
define the Hilbert spaces 
\begin{gather*}
\bH^1_\star(\Omega_F) = \{ \bv\in \bH^1(\Omega_F): \bv|_{\Gamma_F^{\bu}} = \cero\}, \quad 
\bH^1_\star(\Omega_P) = \{ \bw\in \bH^1(\Omega_P): \bw|_{\Gamma_P^{\bd}} = \cero\},\\
 H^1_\star(\Omega_P) = \{ q_P\in H^1(\Omega_P): q_P|_{\Gamma_P^{p_P}} = 0\},
\end{gather*}
and the product space $\bH$
\begin{equation}\label{eq:sol_space}
  \bH = \bH^1_{\star}(\Omega_F)\times \bH^1_{\star}(\Omega_P)\times L^2(\Omega_F)\times L^2(\Omega_P) \times H^1_{\star}(\Omega_P).
\end{equation}
%
Consequently, we have the following weak form for the Biot-Stokes coupling: Find $(\bu, \bd, p_F, \varphi, p_P)\in\bW$ such that 
\begin{subequations} \label{eq:weak}
\begin{align}
 2\mu_f(\beps(\bu),\beps(\bv))_{\Omega_F}
 +\frac{\gamma\mu_f}{\sqrt{\kappa}}\langle T_{\bt} (\bu-\bd),T_{\bt}\bv \rangle_\Sigma & \nonumber\\
 -(p_F, \vdiv \bv)_{\Omega_F}
 + \langle p_P, T_{\nn} \bv \rangle_\Sigma 
 &= F^F(\bv) 
 &\forall \bv  \in \bH^1_{\star}(\Omega_F), \\
  \frac{\gamma\mu_f}{\sqrt{\kappa}}\langle T_{\bt}(\bd-\bu), T_{\bt}\bw\rangle_\Sigma
  + 2\mu_s(\beps(\bd),\beps(\bw))_{\Omega_P} & \nonumber\\ 
  -(\varphi, \vdiv \bw)_{\Omega_P}
  - \langle p_P, T_{\nn} \bw \rangle_\Sigma 
  &= F^P(\bw) &\forall \bw  \in \bH^1_{\star}(\Omega_P),\\
  -(\vdiv \bu, q_F)_{\Omega_F}
  & = 0 & \forall q_F \in  L^2(\Omega_F), \\ 
  \frac{1}{\lambda}(\alpha p_P - \varphi, \psi)_{\Omega_P}
  - (\vdiv \bd, \psi)_{\Omega_P}
  & = 0 & \forall \psi \in L^2(\Omega_P), \\
  - \bigl(C_0+\frac{\alpha^2}{\lambda}\bigr)(p_P, q_P)_{\Omega_P} 
  + \frac{\alpha}{\lambda}(q_P, \varphi)_{\Omega_P} \qquad \quad & \nonumber\\
  + \langle q_P, T_{\nn}(\bu-\bd)\rangle_\Sigma 
  -\frac{\kappa}{\mu_f}(\nabla q_P,\nabla p_P)_{\Omega_P} 
  & = G(q_P) &\forall q_P  \in H^1_{\star}(\Omega_P),
  \end{align}
\end{subequations}
 where 
\begin{gather*} 
  F^F(\bv) = \rho_f( \gg,\bv)_{\Omega_F},  \quad  
  F^P(\bw) = \rho_s(\ff,\bw)_{\Omega_P}, \quad  
  G(q_P) = -(m_P,q_P)_{\Omega_P} - \rho_f(\gg,\nabla q_P)_{\Omega_P} + \rho_f\langle \gg\cdot\nn,q_P\rangle_\Sigma.
\end{gather*}

System \eqref{eq:weak} differs from that analyzed in \cite{ruiz21} in the ordering of the unknowns, and in that we 
obtain a symmetric multilinear formulation defined by a global operator $\mathcal{A}$ (the coefficient 
matrix of the left-hand side of \eqref{eq:weak}) of the form 
\begin{equation}\label{eq:A}
\small{
\left(\begin{array}{cc:ccc}
  -2\mu_f\bdiv \beps + \frac{\gamma\mu_f}{\sqrt{\kappa}} T_{\bt}'T_{\bt} &  -\frac{\gamma\mu_f}{\sqrt{\kappa}} T_{\bt}'                            & \nabla       &                          & T_{\nn}'\\
  -\frac{\gamma\mu_f}{\sqrt{\kappa}} T_{\bt} &\!\!\!\!\!  - 2\mu_s\bdiv\beps + \frac{\gamma\mu_f}{\sqrt{\kappa}} T_{\bt}'T_{\bt}        &         &\!\!\! \!\! \!\! \nabla                   &  -T_{\nn}'\vphantom{\int^{X^X}}  \\[1ex]
  \hdashline
               -\vdiv                                                  &                                                                                &       &                          &\\
                                                                       &   -\vdiv                                                                        &      &\!\!\!\!\! \!\!  - \frac{1}{\lambda} I    & \frac{\alpha}{\lambda} I \\[1ex]
               T_{\nn}                                                  &  -T_{\nn}                                                                      &          &\!\!\!\!\! \!\!  \frac{\alpha}{\lambda} I &\!\!\! \!\!  -\bigl(C_0 + \frac{\alpha^2}{\lambda}\bigr) I + \frac{\kappa}{\mu_f} \Delta  
              \end{array}\right),}
\end{equation}
where the dependence on the model parameters is clearly identified. In particular, the interface coupling terms on
the first off-diagonal blocks depend on the inverse of permeability.

We note that $\mathcal{A}$ can be regarded as defining a perturbed saddle-point problem, with 
\[
\mathcal{A}=\begin{pmatrix}
A & B^{\prime}\\
B & -C\\
\end{pmatrix},
\]
where the composing blocks are defined as
\begin{subequations}
\begin{gather}\label{eq:bs_A}
A = 
\begin{pmatrix}
  \mathcal{A}_{FF} & \mathcal{A}_{FP}\\
  \mathcal{A}_{PF} & \mathcal{A}_{PP}
\end{pmatrix}=
\begin{pmatrix}
  -2\mu_f\bdiv \beps & 0\\
  0 & -2\mu_s\bdiv \beps\\
\end{pmatrix}+
\gamma\frac{\mu_f}{\sqrt{\kappa}}\begin{pmatrix}T_{\bt}'\\ -T_{\bt}'\end{pmatrix}
  \begin{pmatrix}
    T_{\bt}& 
    -T_{\bt}
    \end{pmatrix},\\
\label{eq:bs_BC}
B = \begin{pmatrix}
  -\vdiv & 0\\
  0    & -\vdiv\\
  T_{\nn} & -T_{\nn}\\  
\end{pmatrix},\quad 
C = \begin{pmatrix}
  0 & 0 & 0\\
  0 & \frac{1}{\lambda} I    & -\frac{\alpha}{\lambda} I \\
  0 & -\frac{\alpha}{\lambda} I & \bigl(C_0 + \frac{\alpha^2}{\lambda}\bigr) I - \frac{\kappa}{\mu_f} \Delta
\end{pmatrix}.
\end{gather}\end{subequations}
In turn, well-posedness of the Biot-Stokes system \eqref{eq:weak}
  in the product space \eqref{eq:sol_space} can be established using the abstract Brezzi-Braess theory \cite{braess1996stability}, after invoking separately the solvability and stability 
  results for the Stokes subproblem \cite{girault-raviart} and the Biot subproblem in the
  three-field total pressure formulation \cite{lee17}. However, such a 
  decoupled approach does not lead to stability independent of the
  material parameters and consequently, preconditioners based on the
  standard norms (also referred to as single-physics or sub-physics 
  preconditioners) are not necessarily parameter robust.
  This issue is demonstrated next in Example \ref{ex:naive}.

\begin{example}[Simple preconditioners using standard norms]\label{ex:naive}
  We consider the Biot-Stokes formulation \eqref{eq:weak} defined on the subdomains $\Omega_F=(0, \tfrac{1}{2})\times (0, 1)$,
  $\Omega_P=(\tfrac{1}{2}, 1)\times (0, 1)$ with boundary conditions
  such that the left edge of $\Omega_F$ is a no-slip boundary $\Gamma_F^{\bu}$
  while the top and bottom edges will form $\Gamma^{\bsigma}_F$. Similarly, the top and
  bottom edges on the Biot side are considered stress-free while the right edge
  is clamped.

  Based on the well-posedness of \eqref{eq:weak} in the space $\bH$ (cf. \eqref{eq:sol_space}), 
  we can readily  consider a solution space with weighted inner product leading to the Riesz map (diagonal) preconditioner
  \begin{equation}\label{eq:precond_diagonal}
 \cRD = \begin{pmatrix}
\mathcal{A}_{FF} &  &  & & \\
&\!\!\!\!\! \mathcal{A}_{PP} & && \\
&&\!\!\!\!\! \frac{1}{2\mu_f} I & & \\
&&&\!\!\!\!\!  \bigl(\frac{1}{\lambda}+\frac{1}{2\mu_s}\bigr) I & \\
&&&&\!\!\!\!\!\! \bigl(C_0 + \frac{\alpha^2}{\lambda}\bigr) I - \frac{\kappa}{\mu_f} \Delta  
\end{pmatrix}^{-1}.
  \end{equation}
  We remark that the first and third blocks of $\cRD$ together define a parameter
  robust preconditioner for the (standalone) Stokes problem and the remaining
  blocks form the robust three-field Biot preconditioner \cite{lee17}.
  However, in $\cRD$ the subproblem preconditioners are decoupled.

  Alternatively, after observing that the operator $A$ in \eqref{eq:bs_A}
  defines a norm over the velocity-displacement space $\bH^1_{\star}(\Omega_F)\times \bH^1_{\star}(\Omega_P)$ 
  we will also investigate the (block-diagonal) preconditioner
  \begin{equation}\label{eq:precond_tangent}
 \cRC = \begin{pmatrix}
\mathcal{A}_{FF} & \mathcal{A}_{FP} &  & & \\
\mathcal{A}_{PF} &\!\!\!\!\mathcal{A}_{PP} & && \\
&&\!\!\!\!\! \frac{1}{2\mu_f} I & & \\
&&&\!\!\!\!\!  \bigl(\frac{1}{\lambda}+\frac{1}{2\mu_s}\bigr) I & \\
&&&&\!\!\!\!\!\! \bigl(C_0 + \frac{\alpha^2}{\lambda}\bigr) I - \frac{\kappa}{\mu_f} \Delta
\end{pmatrix}^{-1}.
  \end{equation}
  We note that in \eqref{eq:precond_tangent} the tangential components of
  the Stokes velocity and of the Biot displacement are coupled. In this sense, 
  the preconditioner captures the interaction between the subsystems and, in particular,
  the coupling through the Beavers-Joseph-Saffman condition \eqref{eq:inter-bjs}.

  To investigate the robustness of these preconditioners, we set typical physical parameters
  in \eqref{eq:weak} except for $\mu_f$ and $\kappa$, which are to be varied.
  Using a discretization in terms of the lowest-order Taylor-Hood elements (see
  more details in Section~\ref{sec:discrete}), we next consider the boundedness
  of the number of iterations of the preconditioned MinRes solver under mesh refinement
  and parameter variations. More precisely, using $\cRD$, $\cRC$ (inverted by LU)
  we compare the number of iterations required for convergence 
  determined by reducing the preconditioned residual norm by a factor $10^8$.
  The initial vector is taken as random. 

  We report the results in Table \ref{tab:naive}. It can be seen that the number of MinRes iterations produced with the 
  diagonal preconditioner $\cRD$ is rather sensitive to variations in both $\mu_f$ and $\kappa$.
  In comparison, when fixing $\mu_f=1$, the iterations appear to be more stable 
  in $\kappa$ when \eqref{eq:precond_tangent} is used. However, if $\kappa=1$
  is set, there is a clear deterioration of the performance for small values of $\mu_f$
  also with the preconditioner $\cRC$.
\begin{table}[!t]
  \centering
  \small{
    \begin{tabular}{c|c|cccc||cccc}
      \hline
      & & \multicolumn{4}{c||}{Preconditioner \eqref{eq:precond_diagonal}} & \multicolumn{4}{c}{Preconditioner \eqref{eq:precond_tangent}}\\
      \hline
   $\mu_f$  & \backslashbox{$\kappa$}{$h$} & $2^{-2}$ & $2^{-3}$ & $2^{-4}$ & $2^{-5}$ & $2^{-2}$ & $2^{-3}$ & $2^{-4}$ & $2^{-5}$\\
      \hline
\multirow{3}{*}{$1$} & $10^{-4}$ & 211 & 240 & 258 & 264  & 71 & 82 & 89 & 89\\
& $10^{-2}$ & 76 & 76 & 74 & 73        & 50 & 49 & 48 & 48\\
& $1$      & 41 & 41 & 41 & 41             & 37 & 37 & 36 & 36\\
\hline
      \hline
   $\kappa$  & \backslashbox{$\mu_f$}{$h$} & $2^{-2}$ & $2^{-3}$ & $2^{-4}$ & $2^{-5}$ & $2^{-2}$ & $2^{-3}$ & $2^{-4}$ & $2^{-5}$ \\
      \hline
\multirow{3}{*}{$1$} & $10^{-8}$ & 394 & 184 & -- & -- & 693 & 471 & 639 & --\\                       
&$10^{-2}$ & 59  & 59 & 59 & 58    & 58 & 57 & 55 & 55\\                           
&$1$      & 41  & 41 & 41 & 41    &  37 & 37 & 36 & 36\\                           
      \hline
    \end{tabular}
    }
  
  \smallskip   
  \caption{
    Performance of preconditioners \eqref{eq:precond_diagonal} and \eqref{eq:precond_tangent}
    for the Biot-Stokes system \eqref{eq:weak} in Example \ref{ex:naive}. Only the parameters $\mu_f$ and  $\kappa$ are varied away 
    from 1. The lack of convergence after 750 MinRes iterations is indicated as --.
  }
    \label{tab:naive}
\end{table}

\end{example}

The improved performance of the preconditioner $\cRC$,
  which preserves the tangential coupling of the Stokes and Biot problems
  in \eqref{eq:inter-bjs}, over $\cRD$, where the components are
  decoupled, suggests to strengthening the coupling (involving the tangential traces) in order
  to obtain parameter robustness. However, from the point
  of view of the interface conditions \eqref{eq:inter-U}-\eqref{eq:inter-bjs},
  it is clear that the coupling in the normal direction is missing in $\cR_{C}$.

  With the above idea in mind, we proceed to establish well-posedness of \eqref{eq:weak}
  in the product space equipped with non-standard norms that include additional
  control at the interface reflecting/arising from the mass conservation
  condition \eqref{eq:inter-U}.

\section{Well-posedness of the Biot-Stokes  system}\label{sec:analysis}
Let us group the variables as $\vec{\bu}= (\bu,\bd)$ and $\vec{p} = (p_F, \varphi, p_P)$ and introduce the weighted norm
\begin{equation}\label{eq:product-norm}
\|(\vec{\bu}, \vec{p}) \|_{\bH_\epsilon}^2 : = 
\| \vec{\bu} \|_A^2
+ | \vec{p} |_B^2
+ | \vec{p} |_C^2,
\end{equation}
with
\begin{subequations}
\begin{align}
  \| \vec{\bu} \|_A^2
  &:= 2\mu_f \|\beps(\bu)\|^2_{0,\Omega_F} 
  + \sum_{j=1}^{d-1} \frac{\gamma\mu_f}{\sqrt{\kappa}} \|(\bu-\bd)\cdot \bt_j\|^2_{0,\Sigma} 
  + 2\mu_s \|\beps(\bd)\|^2_{0,\Omega_P}, \label{normA}\\
  | \vec{p} |_B^2 
  &:= 
  \frac{1} { {2\mu_f}} \|p_F\|^2_{0,\Omega_F} 
  + \frac{1}{2\mu_s}\|\varphi\|^2_{0,\Omega_P}
  + \left( \frac{1}{2\mu_f} + \frac{1}{2\mu_s} \right) {\| p_P|_\Sigma \|^2_{-\frac12,01,\Sigma}}, \label{normB}\\
  | \vec{p} |_C^2 
  &:= 
  \frac{1}{\lambda} \|\varphi-\alpha p_P\|^2_{0,\Omega_P}
  + C_0 \|p_P\|^2_{0,\Omega_P} + \frac{\kappa}{\mu_f} \|\nabla p_P\|^2_{0,\Omega_P},\label{normC}
\end{align}
\end{subequations}
where the fractional norm is defined in \eqref{eq:norm-H0p-dual}.

In turn, we define the weighted product space $\bH_\epsilon$ as the space that contains all $(\vec \bu, \vec p)$ that are bounded in this norm.
The subscript $\epsilon$ encodes the collection of weighting parameters  $\kappa,\alpha,\gamma,\mu_f,\mu_s,C_0,\lambda$. Moreover, the space allows for the natural decomposition:
\[
\bH_\epsilon = \vec \bV \times \vec Q.
\]

\begin{theorem}\label{theo:iso}
Problem \eqref{eq:weak} is well-posed in the space $\bH_\epsilon$ equipped with the norm \eqref{eq:product-norm}. In other words, the operator $\cA: \bH_\epsilon\to \bH_\epsilon'$ in \eqref{eq:A} 
is a symmetric isomorphism satisfying 
\begin{subequations}
\begin{align}
 \| \cA\|_{\mathcal{L}(\bH_\epsilon, \bH_\epsilon')} \leq C_1, \label{continuity}\\
 \| \cA^{-1}\|_{\mathcal{L}(\bH_\epsilon', \bH_\epsilon)} \leq C_2, \label{invertibility}
\end{align}\end{subequations}
where $C_1,C_2$ are positive constants independent of $\epsilon$. 
\end{theorem} 
\begin{proof}
The operator norm is defined as
	\[
	\|\cA\|_{\mathcal{L}(\bH_{\beps},\bH_{\beps}')} := 
  \sup_{(\vec{\bu},\vec{p}), (\vec{\bv},\vec{q})}\frac{\langle\cA(\vec{\bu},\vec{p}),(\vec{\bv},\vec{q})\rangle}{\|(\vec{\bu},\vec{p})\|_{\bH_{\beps}}||(\vec{\bv},\vec{q})||_{\bH_{\beps}'}},
	\]
and condition \eqref{continuity} states the continuity of $\cA$. To show this, we first use the Cauchy-Schwarz inequality to derive
\begin{align}
  \langle A \vec \bu, \vec \bv \rangle &\le \| \vec \bu \|_A \| \vec \bv \|_A, &
  \langle C \vec p, \vec q \rangle &\le | \vec p |_C | \vec q |_C.\label{eq:boundAC}
\end{align}
It therefore remains to show that $B$ is continuous. Another application of the Cauchy-Schwarz inequality on the different terms together with a trace inequality provides this result. 

In order to prove the second relation \eqref{invertibility}, we aim to verify the assumptions of the Banach-Ne\v{c}as-Babu\v{s}ka (BNB) theorem (see, e.g., \cite{ern04}). In particular, we aim to prove that
\begin{align} \label{eq: BNB}
  \sup_{(\vec \bv, \vec p)}
  \frac{\langle \cA(\vec{\bu},\vec{p}),(\vec{\bv}, \vec q) \rangle }
  {\|(\vec{\bv}, \vec{q}) \|_{\bH_\epsilon}}
  &\gtrsim \|(\vec{\bu}, \vec{p}) \|_{\bH_\epsilon}, 
  & \forall (\vec{\bu}, \vec{p})&\in \bH_\epsilon.
\end{align}

We do this by assuming that $(\vec \bu, \vec p) \in \bH_\epsilon$ is given and by constructing an appropriate test function $(\vec \bv, \vec q) \in \bH_\epsilon$. Following, e.g., \cite{anaya21,lee17}, we choose $\vec{q}= - \vec{p}$ and $\vec{\bv}=\vec{\bu}$, giving 
\begin{align}\label{eq:step1}
  \langle \cA(\vec{\bu},\vec{p}),(\vec{\bu},-\vec{p}) \rangle &= 
  \langle A \vec{\bu}, \vec \bu \rangle
  + \langle B \vec{\bu}, \vec p \rangle
  - \langle B \vec{\bu}, \vec p \rangle
  + \langle C \vec p, \vec p \rangle \nonumber \\
  &=
  \| \vec{\bu} \|_A^2
  + | \vec{p} |_C^2. 
\end{align}

Next, the inf-sup condition proven in Lemma~\ref{lem: inf-sup B}, below, allows us to construct $\vec \bv_p$ such that
\begin{align} 
\langle B \vec \bv_p, \vec p \,\rangle
  &= | \vec p |_B^2, &
  \| \vec \bv_p \|_A 
  &\le \beta_0^{-1} | \vec p |_B.\label{eq:boundB}
\end{align}
We now use this test function, scaled by a constant $\delta > 0$ to be chosen later, and using \eqref{normA}-\eqref{normC} along with \eqref{eq:boundAC} and \eqref{eq:boundB}, we derive 
\begin{align*}
 \langle \cA(\vec{\bu}, \vec{p}),(\delta \vec{\bv}_p, \vec0) \rangle &= 
  \langle A \vec{\bu}, \delta \vec \bv_p \rangle
  + \delta [ B \vec \bv_p, \vec p ] 
  \nonumber \\
  &\ge - \delta \| \vec{\bu} \|_A \| \vec \bv_p \|_A
  + \delta | \vec p |_B^2
  \nonumber \\
  &\ge - \frac12 \| \vec{\bu} \|_A^2 
  - \frac12 \delta^2 \| \vec \bv_p \|_A^2 
  + \delta | \vec p |_B^2
  \nonumber \\
  &\ge - \frac12 \| \vec{\bu} \|_A^2 
  + (\delta - \frac12 \beta_0^{-2} \delta^2) | \vec p |_B^2,
\end{align*}
where we have also used Cauchy-Schwarz and Young's inequality. Setting $\delta = \beta_0^2$ gives us
\begin{align} \label{eq:step2}
  \langle \cA(\vec{\bu}, \vec{p}),(\delta \vec{\bv}_p, \vec0) \rangle
  &\ge - \frac12 \| \vec{\bu} \|_A^2 
  + \frac12 \beta_0^2 | \vec p |_B^2.
\end{align}

Finally, we take $(\vec \bv, \vec q) = (\vec \bu + \delta \vec \bv_p, - \vec p)$ and put together  \eqref{eq:step1} and \eqref{eq:step2} to arrive at 
\begin{align*}
  \langle \cA(\vec{\bu},\vec{p}),(\vec{\bv}, \vec q) \rangle
  &\ge
  \frac12 \| \vec{\bu} \|_A^2 
  + \frac12 \beta_0^2 | \vec p |_B^2
  + | \vec p |_C^2 \nonumber \\
  &\gtrsim \|(\vec{\bu}, \vec{p}) \|_{\bH_\epsilon}^2, \\
  \|(\vec{\bv}, \vec{q}) \|_{\bH_\epsilon}^2
  &\le 2 \left(\|(\vec{\bu}, \vec{p}) \|_{\bH_\epsilon}^2 + \delta^2 \| \vec \bv_p \|_A^2 \right) \nonumber \\
  &\le 2 \left(\|(\vec{\bu}, \vec{p}) \|_{\bH_\epsilon}^2 + \beta_0^2 | \vec p |_B^2 \right) \nonumber \\
  &\lesssim \|(\vec{\bu}, \vec{p}) \|_{\bH_\epsilon}^2.
\end{align*}
The combination of these two bounds shows that \eqref{eq: BNB} holds. The BNB theorem now provides \eqref{invertibility}.
\end{proof}

\begin{lemma} \label{lem: inf-sup B}
There exists a $\beta_0 > 0$ such that 
\[
\sup_{\cero \ne \vec{\bv}}
\frac{\langle B \vec \bv, \vec p \rangle}{\|\vec \bv\|_A}
\geq \beta_0 | \vec p |_B \qquad \forall \vec p \in \vec Q.
\]
\end{lemma}

\begin{proof} The proof  follows similarly to  \cite[Section 3]{boon21SD}. 
Let $\vec{p} = (p_F, \varphi, p_P) \in \vec Q$ be given. We proceed in five steps. 

\begin{enumerate}
  \item 
With a given total pressure $\varphi$ in the Biot domain, we set up an auxiliary Stokes problem: Find $(\bz_0, s_0) \in \bH^1(\Omega_P) \times L^2(\Omega_P)$ that weakly satisfy
\begin{subequations}
\begin{align*}
- \bdiv \left( \beps(\bz_0) + s_0 \bI \right) &= \cero,\\
\vdiv \bz_0 &= -\varphi & \text{in} \quad &\Omega_P,
\end{align*}
subject to the mixed boundary conditions
\begin{align*}
\bz_0 &= \cero \ \text{on} \quad  \Gamma_P^{\bd} \cup \Sigma, & \text{and} \quad 
\left( \beps(\bz_0) + s_0 \bI \right) \nn &= \cero \ \text{on} \ \Gamma_P^{p_P}.
\end{align*}
\end{subequations}
By the well-posedness of this auxiliary problem (for a proof see, e.g., \cite[Chapter I]{girault-raviart}), the first component of the solution satisfies  
\[
\|\bz_0 \|_{1,\Omega_P} \lesssim \| \varphi \|_{0, \Omega_P}.
\]
  \item
We next consider that  the trace of $p_P$ is a distribution in $H^{-1/2}(\partial\Omega_P)$, and focus on its restriction to the interface, $p_P|_{\Sigma}$, belonging to $[H_{01}^{1/2}(\Sigma)]'$ (cf., the end of Section~\ref{sec:intro}).

 Let $\zeta \in H_{01}^{1/2}(\Sigma)$ be the Riesz representative of $p_P|_{\Sigma} \in [H_{01}^{1/2}(\Sigma)]'$, and 
 consider a  Stokes-extension of $\zeta$ into $\Omega_P$ by setting up  another auxiliary Stokes problem (and still in the Biot domain): Find a pair $(\bz_1, s_1) \in \bH^1(\Omega_P) \times L^2(\Omega_P)$ that weakly satisfies 
\begin{subequations}\label{aux:p2}
\begin{align}
- \bdiv \left( \beps(\bz_1) + s_1 \bI \right) &= \cero,\\
\vdiv \bz_1 &= 0 & \text{in} \quad &\Omega_P, \nonumber
\end{align}
subject to the mixed-type boundary conditions
\begin{align}
\bz_1 &= \zeta \nn \ \text{on} \ \Sigma, \\
\bz_1 &= \cero \ \text{on} \ \Gamma_P^{\bd}, &
\left( \beps(\bz_1) + s_1 \bI \right) \nn &= \cero \ \text{on} \ \Gamma_P^{p_P}.\nonumber
\end{align}
\end{subequations}

Again, we use the well-posedness of the auxiliary problem (in this case, \eqref{aux:p2})  to conclude that the  extension function $\bz_1$ satisfies a continuous dependence on data 
\[
\|\bz_1 \|_{1,\Omega_P} \lesssim \| \zeta \|_{\frac12,01, \Sigma}  
= \| p_P \|_{-\frac12,01, \Sigma}.
\]

\item
We combine the two previous steps to form a \emph{Biot velocity} as $\bv_P := \frac{1}{2 \mu_s}(\bz_0 + \bz_1)$. By construction, this function has the following properties
\begin{align*}
  \langle B (\bv_P, \cero), \vec p \rangle
&= - (\vdiv \bv_P, \varphi)_{\Omega_P}
+ \langle\nn \cdot \bv_P, p_P\rangle_{\Sigma} \nonumber \\
&= \frac{1}{2 \mu_s} \left(- (\vdiv \bz_0, \varphi)_{\Omega_P}
  + \langle\nn \cdot \bz_1, p_P\rangle_{\Sigma}\right) \nonumber \\
 &= \frac{1}{2 \mu_s} \left(\| \varphi \|_{\Omega_P}^2 + \| p_P|_\Sigma \|_{-\frac12,01,\Sigma}^2 \right), \\
 \| (\bv_P, \cero) \|_A^2 
 &= 2\mu_s \|\beps(\bv_P)\|^2_{0,\Omega_P} \nonumber \\
 &\le \frac{1}{2 \mu_s} 2 \left( \| \bz_0 \|^2_{1,\Omega_P} + \| \bz_1 \|^2_{1,\Omega_P} \right) \nonumber \\
 &\lesssim \frac{1}{2 \mu_s} \left(\| \varphi \|_{\Omega_P}^2 + \| p_P|_\Sigma \|_{-\frac12,01,\Sigma}^2 \right).
\end{align*}

\item
Next, we repeat the first three steps with $p_F$ substituted for $\varphi$, $-\zeta$ for $\zeta$, $\Omega_F$ instead of $\Omega_P$ (as well as the relevant boundaries), and $\mu_f$ substituted for $\mu_s$. In particular, for step 2 we note that the Riesz representative $\zeta$ is constructed relatively to $\partial\Omega_P$, but the trace of $\bH_\star^1(\Omega_F)$ can be regarded as equivalent to the trace of $\bH_\star^1(\Omega_P)$ provided that the shapes and measures of each subdomains are similar (see \cite{galvis07,layton02}, where the argument is applied to a normal velocity trace on the interface). Then one can also consider $\zeta \in H_{01}^{1/2}(\Sigma)$, as relative  to $\partial\Omega_F$. 

In this case we end up with  a \emph{Stokes velocity}   $\bv_F$ satisfying the relations
\begin{align*}
\langle B (\bv_F,\cero), \vec p \rangle
&= \frac{1}{2 \mu_f} \left(\| p_F \|_{\Omega_F}^2 +\| p_P|_\Sigma \|_{-\frac12,01,\Sigma}^2 \right), \\
\| (\bv_F,\cero) \|_A^2 
&\lesssim \frac{1}{2 \mu_f} \left(\| p_F \|_{\Omega_F}^2 + \| p_P|_\Sigma \|_{-\frac12,01,\Sigma}^2 \right).
\end{align*}

\item For the final step it suffices to combine steps 3 and 4, and choose as test function the pair of functions constructed above $\vec \bv = (\bv_F, \bv_P)$, which leads to 
\begin{align} \label{eq: final properties}
 \langle B \vec \bv, \vec p \,\rangle
  &= | \vec p |_B^2, &
  \| \vec \bv \|_A^2 
  &\lesssim | \vec p |_B^2.
\end{align}
\end{enumerate}
Equation \eqref{eq: final properties} also contains the right scaling  through the use of the $|\cdot|_B$ seminorm. This step concludes the proof.
\end{proof}


According to the general operator preconditioning framework from \cite{mardal-winther}, 
Theorem~\ref{theo:iso} yields that 
a parameter-robust preconditioner can be constructed based on the 
Riesz map $\cRF :  \bH_\epsilon'\to \bH_\epsilon$ satisfying 
\[ \|\cRF\|_{\mathcal{L}(\bH_\epsilon', \bH_\epsilon)} \leq 1,\quad \|\cRF^{-1}\|_{\mathcal{L}(\bH_\epsilon, \bH_\epsilon')} \leq 1, \]
which implies that 
\[\mathrm{cond}(\cRF \mathcal{A}) = \|\cRF\mathcal{A}\|_{\mathcal{L}(\bH_\epsilon, \bH_\epsilon)} \|(\cRF\mathcal{A})^{-1}\|_{\mathcal{L}(\bH_\epsilon, \bH_\epsilon)} \leq C_1C_2.\]
Following Theorem~\ref{theo:iso} a natural block-diagonal preconditioner for the Biot-Stokes problem is
therefore the Riesz map with respect to the inner product in $\bH_\epsilon$ 
  \begin{equation}\label{eq:robust_precond}
    \cRF = 
      \left(\begin{array}{cc:ccc}
\!\!\!\!\mathcal{A}_{FF} & \!\!\mathcal{A}_{FP} \!\!&  & &\\
\!\!\!\!\mathcal{A}_{PF} & \!\!\mathcal{A}_{PP} \!\!& && \\[0.5ex]
\hdashline\\[-1.5ex]
&& \!\frac{1}{2\mu_f} I & & \\
&&& \!\!\!\!\!\!\bigl(\frac{1}{\lambda}+\frac{1}{2\mu_s}\bigr) I & \!\!\!\!\!-\frac{\alpha}{\lambda}I\\
&&&\!\!\!\!\!\!-\frac{\alpha}{\lambda}I & \!\!\!\!\!\!\!\bigl(C_0 + \frac{\alpha^2}{\lambda}\bigr) I - \frac{\kappa}{\mu_f} \Delta  + \frac{1}{\mu}(-\Delta_{\Sigma, 01})^{-\frac12}\!\!\!\!
\end{array}\right)^{-1}\!\!,
\end{equation}
  where we have defined $\mu^{-1}:=(2\mu_s)^{-1}+(2\mu_f)^{-1}$. We remark
  that the fractional operator $(-\Delta_{\Sigma, 01})^{-1/2}$ induces a norm
  on the space $[H^{1/2}_{01}(\Sigma)]'$  (see also \cite{arioli2009discrete} for the case of $[H^{1/2}_{00}(\Sigma)]'$).

\section{Discretization and robust preconditioners for Biot-Stokes}\label{sec:results}
\subsection{Preliminaries and accuracy verification}\label{sec:discrete}
In order to define a finite element method for the Biot-Stokes system \eqref{eq:weak} we will use two families of Stokes-stable elements. The 
generalized
Taylor-Hood ($\text{TH}_k$) type for the pairs [fluid velocity, fluid pressure] and [porous displacement, total pressure], while using continuous and 
piecewise polynomials of degree $k+1$ for the 
porous fluid pressure; 
as well as 
non-conforming discretizations based on the lowest-order Crouzeix-Raviart (CR) elements for the pairs  [fluid velocity, fluid pressure] and [porous displacement, total pressure] (using for velocity and displacement the facet stabilization described in \cite{burman05}, see also \eqref{eq:CR-biot}),  and continuous and piecewise linear elements for the Biot fluid pressure. 
A requirement is that the meshes for the Biot and Stokes subdomains match at the interface. For sake of completeness, the precise definition of the finite element subspaces is given in \ref{section:TH-CR}.
We remark that the non-conforming CR element discretization is not covered by the theory presented in this paper and is as such an interesting test case.

For the numerical realization of the methods discussed above, we have used the open source finite element library \texttt{FEniCS} \cite{alnaes15,logg12}, as well as the specialised module \texttt{FEniCS$_{ii}$} \cite{fenicsii} for handling subdomain- and boundary- restricted terms and variables.

We verify the error decay using the TH$_k$ spaces with polynomial degrees $k =1, 2$, and the
CR family. For this we simply use unity parameters. We consider synthetic forcing terms and
boundary data such that the exact manufactured solutions to \eqref{problem} are
\begin{equation}\label{eq:mms}
  \begin{aligned}
\bu &= \begin{pmatrix}
\cos(\pi x)\sin(\pi y)\\-\sin(\pi x)\cos(\pi y)\end{pmatrix}, \,
p_F = \exp(xy)+\cos(\pi x)\cos(\pi y),\\
\bd &= \begin{pmatrix} \cos(\pi x)\sin(\pi y)+\frac{y(x-0.5)}{\lambda}\\-\sin(\pi x)\cos(\pi y)\end{pmatrix}, \,
p_P = \cos(\pi(x^2+y^2)), \, \varphi = \alpha p_P - \lambda\vdiv\bd.
\end{aligned}
\end{equation}
Note that these exact solutions require non-homogeneous transmission conditions. 

We construct a series of uniformly successively refined triangular meshes for $\Omega=(0,1)^2$, defining the interface as the  segment $\{0.5\}\times (0,1)$ and considering the left half of the domain as $\Omega_F$ and the right half as $\Omega_P$. Then, we proceed to measure individual errors between closed-form and approximate solutions in the  usual norms  
\begin{gather*}
e(\bu) =\|\bu-\bu_h\|_{1,\Omega_F},\quad  
e(p_F) =\|p_F-p_{F,h}\|_{0,\Omega_F},\\
 e(\bd) =\|\bd-\bd_h\|_{1,\Omega_P}, \quad 
 e(\varphi) =\|\varphi-\varphi_h\|_{0,\Omega_P},\quad 
 e(p_P)=\|p_P-p_{P,h}\|_{1,\Omega_P}.\end{gather*}

Figure \ref{fig:cvrg} reports the approximation errors for the three discretizations. In all cases the expected order $k+1$ for $\text{TH}_k$ can be observed. For the CR family we obtain the expected linear convergence.

\begin{figure}[t!]
  \includegraphics[width=0.32\textwidth]{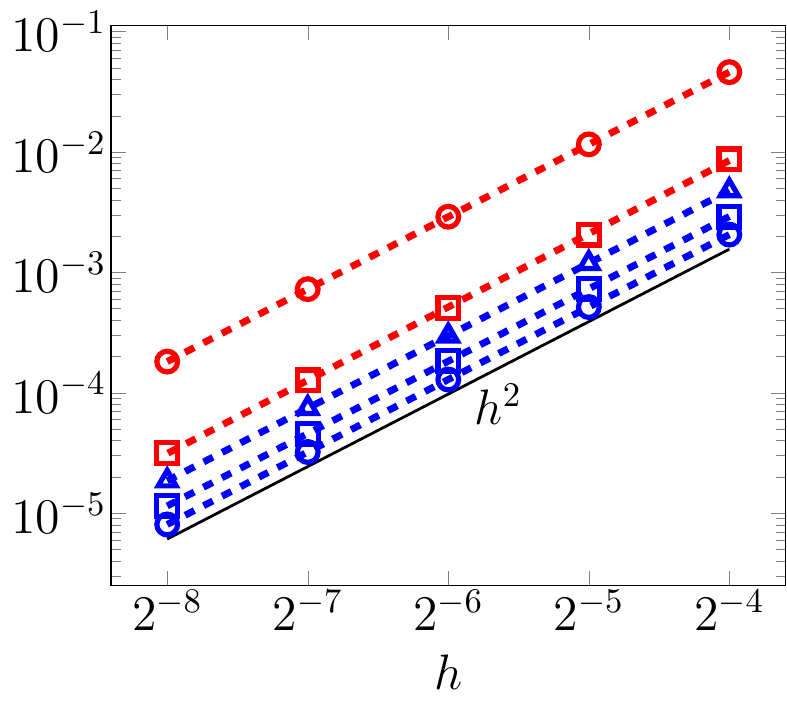}
  \includegraphics[width=0.32\textwidth]{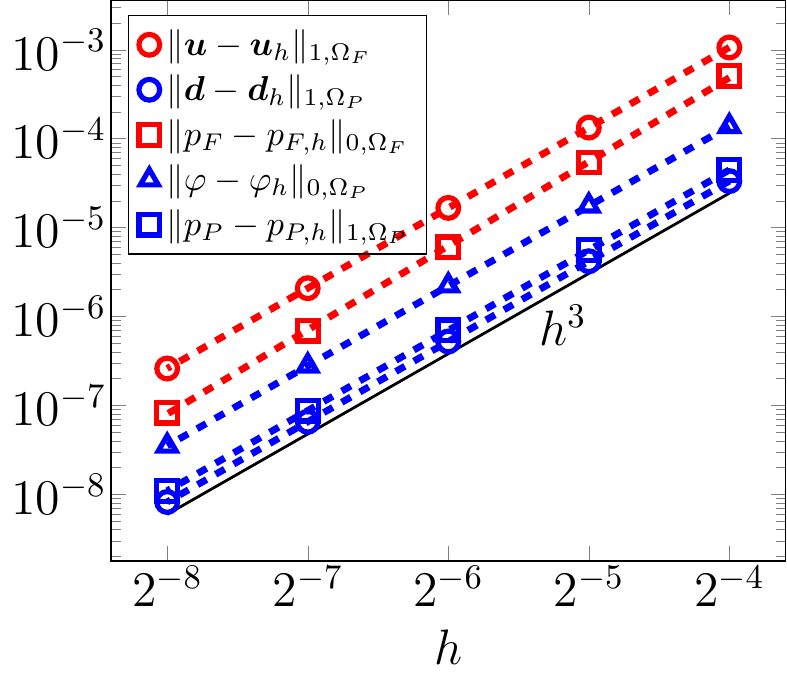}
  \includegraphics[width=0.32\textwidth]{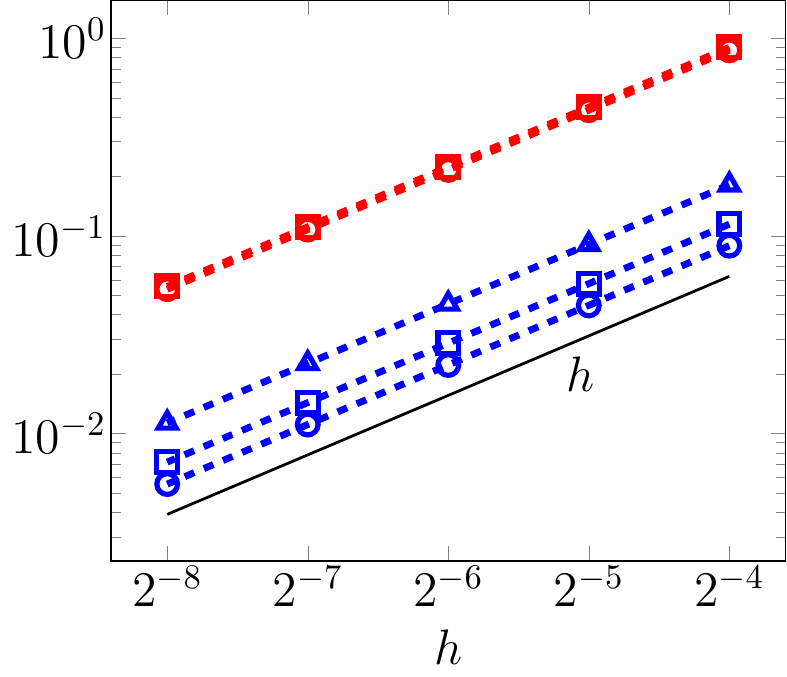}  
  \label{fig:cvrg}
  \vspace{-10pt}
  \caption{Error convergence of the finite element approximation 
    of \eqref{eq:weak} with the manufactured solution \eqref{eq:mms} using (left) $\text{TH}_1$, (center)
    $\text{TH}_2$ and (right) CR families for discretization. Optimal convergence is observed
    in all cases. The legend is shared between the subplots.
  }
\end{figure}

Having defined suitable finite element discretization for the Biot-Stokes
system, we next investigate robustness of the fractional preconditioner
\eqref{eq:robust_precond} which was established theoretically in Theorem
\ref{theo:iso} with the assumption of specific boundary conditions
on the sub-boundaries intersecting the interface, namely, that $\Sigma$ meets the intersection between the Biot displacement and the Stokes velocity boundaries. However,
the theory and in turn the $\cRF$ preconditioners can be extended to more general
boundary conditions as we will demonstrate by the numerical experiments. In
particular, in Section \ref{sec:num_robust_neumann} we consider the setup
where the interface intersects boundaries $\Gamma^{\bsigma}_F$, $\Gamma^{p_P}_P$
see Figure \ref{fig:sketch}. Then, in Section \ref{sec:num_brain} the interface
is a closed curve.

Due to the Laplace operator on the interface, the discretization of
preconditioners $\cRF$ \eqref{eq:robust_precond} is not immediately evident.
Before discussing the results let us therefore comment on the construction of the critical
component, that is, the fractional operators.

\subsection{Discrete preconditioner}\label{sec:discrete_precond} From \eqref{eq:robust_precond} we observe 
that in the pressure block of the preconditioner $\cRF$ the operator acting on $(\varphi, p_P)$,
\begin{equation}\label{eq:pressure_precond}
\begin{pmatrix}\bigl(\frac{1}{\lambda}+\frac{1}{2\mu_s}\bigr) I & -\frac{\alpha}{\lambda}I\\
  -\frac{\alpha}{\lambda}I & \bigl(C_0 + \frac{\alpha^2}{\lambda}\bigr) I - \frac{\kappa}{\mu_f} \Delta  + \mu^{-1}(-\Delta_{\Sigma, 01})^{-\frac12}\\
\end{pmatrix}^{-1},
\end{equation}
contains a sum of a bulk term coupling the two pressures and a fractional interface term
$\mu^{-1}(-\Delta_{\Sigma, 01})^{-1/2}$. Thus, the action of the operator implicitly 
involves the trace of the Biot pressure at the interface $\Sigma$. In the discrete
setting mirroring this property requires a discrete trace space $S_h$ and a
restriction operator. In the following we choose $S_h$ 
as the space of piecewise continuous polynomials of order $k+1$ whenever the $\text{TH}_k$ family is used.
For the CR family, $S_h$ is constructed from piecewise continuous linear functions.
The restriction operator is then realized as an $L^2$-projection.

Once on the interface we approximate the fractional operator based on the
spectral decomposition, see, e.g., \cite{kuchta2016preconditioners}. That is,
\begin{equation}\label{eq:fract_laplace}
\langle \mu^{-1}(-\Delta_{\Sigma, 01})^{-1/2}u, v \rangle_{\Sigma} := \sum_{i}\lambda^{-1/2}_i(\mu^{-1} u_i, u)_{\Sigma}(\mu^{-1} u_i, v)_{\Sigma} \quad \forall u, v \in S_h,
\end{equation}
where $(u_i, \lambda_i)\in S_h\times\mathbb{R}$ are solutions of the generalized
eigenvalue problem 
\begin{equation}\label{eq:eigv}
(\mu^{-1} \nabla u_i, \nabla v)_{\Sigma}=\lambda_i (\mu^{-1}u_i, v)_{\Sigma}\quad\forall v\in S_h,
\end{equation}
satisfying the orthogonality condition $(\mu^{-1}u_i, u_j)_{\Sigma}=\delta_{ij}$. 
Note that in \eqref{eq:eigv} the Dirichlet boundary conditions are prescribed on
$\partial\Sigma$ reflecting the trace spaces of velocity and displacement
when $\Gamma^{\bu}_F$ and $\Gamma^{\bd}_P$ are incident to $\Sigma$ (as was
assumed in Theorem \ref{theo:iso}). 

Going beyond the theoretical analysis, we demonstrate in Section \ref{sec:num_robust_neumann}
that for the configuration with $\Gamma^{\bsigma}_F$, $\Gamma^{p_P}_P$ intersected by the interface,
the operator \eqref{eq:pressure_precond} (and in turn the preconditioner \eqref{eq:robust_precond})
needs to be modified. Specifically, the fractional term then reads
$\mu^{-1}(-\Delta_{\Sigma}+I_{\Sigma})^{-1/2}$. The operator is defined analogously
to \eqref{eq:fract_laplace}, where, in contrast, the $H^1$-norm 
(cf. the $H^1$-seminorm in \eqref{eq:eigv}) is now used in the eigenvalue problem:
Find $(u_i, \lambda_i)\in (S_h, \mathbb{R})$ such that 
\begin{equation}\label{eq:eigv_neumann}
  (\mu^{-1} \nabla u_i, \nabla v)_{\Sigma} + (\mu^{-1} u_i, v)_{\Sigma}=\lambda_i (\mu^{-1}u_i, v)_{\Sigma}\quad\forall v\in S_h
\end{equation}
and $(\mu^{-1} u_i, u_j)_{\Sigma}=\delta_{ij}$. Note that here the Neumann boundary
conditions are prescribed on $\partial\Sigma$.

We remark that the fractional operators and in particular the boundary
conditions in \eqref{eq:eigv} and \eqref{eq:eigv_neumann} must be set based
on the configuration of the boundaries with respect to the interface. The
fact that the conditions cannot be chosen freely is investigated next in Example 
\ref{rmrk:nitsche} together with the observation that parameter stability of
the preconditioners \eqref{eq:robust_precond} is affected by enforcement
of Dirichlet boundary conditions in construction of the fractional operators
via the eigenvalue problems \eqref{eq:eigv} and \eqref{eq:eigv_neumann}.

\begin{example}[Boundary conditions in fractional operators]\label{rmrk:nitsche}
  In the following, given the geometry of Example \ref{ex:naive}, let $C_0=0$, $\kappa=10^{-10}$
  while the remaining problem parameters of the the Biot-Stokes system
  \eqref{eq:weak} are set to unity. This choice is made to put emphasis on the fractional
  term in the preconditioner \eqref{eq:robust_precond}.

  Assuming first that $\Sigma$ intersects $\Gamma^{\bsigma}_F$ and $\Gamma^{p_P}_P$,
  we consider \eqref{eq:weak} either with a preconditioner \eqref{eq:robust_precond},
  where Dirichlet boundary conditions are (strongly) enforced on the fractional operator, or
  a modified preconditioner which uses the operator $\mu^{-1}(-\Delta_{\Sigma}+I_{\Sigma})^{-1/2}$
  constructed with the Neumann boundary conditions, see \eqref{eq:eigv_neumann}.
  Using a discretization by $\text{TH}_1$ we observe in Table \ref{tab:nitsche} that 
  the Dirichlet conditions result in a lack of boundedness in the mesh size. On the
  other hand, when Neumann boundary conditions are imposed on the fractional operator the
  spectral condition numbers of the preconditioned problem seem to converge along with mesh 
  refinement.

  Repeating the experiment for the configuration where the interface is incident
  to $\Gamma^{\bu}_F$ and $\Gamma^{\bd}_P$, it can be seen that Neumann boundary conditions
  lead to a growth similar to what was observed in the previous setup with Dirichlet datum.
  Furthermore, in Table \ref{tab:nitsche} the condition numbers blow up also with
  $\mu^{-1}(-\Delta_{\Sigma, 01})^{-1/2}$. However, in this case the growth can be 
  traced to the (two) eigenvalues that correspond to the degrees of freedom\footnote{
    The number of unbounded modes is finite (and independent of refinement)
    when $\Sigma$ is a curve. However,
    when the interface is a manifold in 3$d$ the number of unbounded modes
    grows with $h$ (as $\partial\Sigma$ is refined).
  }
  of the space $S_h$ on $\partial\Sigma$ which are set strongly by the Dirichlet boundary conditions.
  This observation motivates using the Nitsche technique \cite{Nitsche1971berEV} in
  order to enforce the conditions on $\mu^{-1}(-\Delta_{\Sigma, 01})^{-1/2}$. 
  With a suitably chosen Nitsche penalty parameter, Table \ref{tab:nitsche} (column D\textsubscript{Nitsche})
  reveals that \eqref{eq:robust_precond} yields mesh independence. 
  
  The fact that for $\Sigma$ intersecting $\Gamma^{\bu}_F$ and $\Gamma^{\bd}_P$
  the numerical issues with Dirichlet boundary conditions of the fractional operators are
  related to their strong enforcement, can be further illustrated
  using a discretization for which the intermediate trace space $S_h$ has no
  degrees of freedom on the boundary $\partial\Sigma$. To this end, we here  consider a modification of the
  $\text{CR}$ family where the Biot fluid pressure space is made of piecewise constants.
  The discrete Laplace operator in \eqref{eq:weak} as well as in the eigenvalue
  problem \eqref{eq:eigv} is then defined analogously to the finite volume method, 
  and in particular utilizing two-point flux approximation, see, e.g., \cite{Droniou2017}.
  After using $\mu^{-1}(-\Delta_{\Sigma, 01})^{-1/2}$ with Dirichlet boundary conditions enforced
  weakly, stable condition numbers are achieved, as observed in column $\text{CR}^{\sharp}$ of Table
  \ref{tab:nitsche}.

  In order to show that strong enforcement of the Dirichlet boundary conditions can be used depending
  on the boundary configuration, we finally consider a setup where the interface
  meets $\Gamma^{\bu}_F$ on the Stokes side while on the incident Biot boundary (which we denote by  $\Gamma^{\dagger}_P$) 
  we assume Dirichlet data on the displacement $\bd$ and on the pressure $p_P$.
  Then, using a 
  $\text{TH}_1$ discretization together with the preconditioner \eqref{eq:robust_precond}, bounded
  condition numbers are produced, which can be observed in the last column of Table \ref{tab:nitsche}.
  
\begin{table}[t!]
  \centering
  \footnotesize{
    \begin{tabular}{c|cc||ccc|c||c}
    \hline
    \multirow{2}{*}{$\log_2 h^{-1}$} & \multicolumn{2}{c||}{$\Gamma^{\bsigma}_F, \Gamma^{p_P}_P$} & \multicolumn{4}{c||}{$\Gamma^{\bu}_F, \Gamma^{\bd}_P$} & $\Gamma^{\bu}_F, \Gamma^{\dagger}_P$\\[0.5ex]
    \cline{2-8}
    & D & N & D & N & D\textsubscript{Nitsche} & D with {$\text{CR}^{\sharp}$} & D\\
    \hline      
2 & 10.61 & 16.67 & 3369   &24.48 &7.02 &8.07 & 6.77\\
3 & 12.17 & 17.58 & 13879  &30.29 &7.43 &8.45 & 7.31\\
4 & 13.90 & 18.12 & 56254  &35.91 &7.59 &8.60 & 7.53\\
5 & 15.73 & 18.53 & --     &41.66 &7.67 &8.60 & 7.64\\
6 & 17.63 & 18.83 & --     &47.66 &7.71 &8.57 & 7.69\\
7 & 19.58 & 19.06 & --     &53.96 &7.72 &8.54 & 7.71\\
\hline
  \end{tabular}
  }
  \smallskip 
  \caption{
    Spectral condition numbers for the Biot-Stokes problem \eqref{eq:A} with
    fractional preconditioners \eqref{eq:robust_precond}. Boundary condition
    configurations from Figure \ref{fig:sketch} are considered; $\Sigma$
    intersects $\Gamma^{\bsigma}_F$ and $\Gamma^{p_P}_P$ or $\Gamma^{\bu}_F$ and $\Gamma^{\bd}_P$.
    In addition, on $\Gamma^{\dagger}_P$ we prescribe both $\bd$ and $p_P$.
    The fractional preconditioners differ by the boundary conditions enforced
    on $\partial\Sigma$; Dirichlet condition, cf. \eqref{eq:eigv}, enforced 
    strongly (D) or with Nitsche's method (D\textsubscript{Nitsche}) or Neumann condition (N),
    cf. \eqref{eq:eigv_neumann}. Systems are discretized by $\text{TH}_1$ family except
    for $\text{CR}^{\sharp}$ where the \emph{modified} $\text{CR}$ family is used with
    the Biot pressure space constructed from piecewise constant functions. For the
    finest refinement level the discrete linear system contains approximately 350 thousand degrees of freedom.
  }
  \label{tab:nitsche}
\end{table}
\end{example}

We finally remark that the spectral realization \eqref{eq:fract_laplace} is not scalable
to problems where the interface and the trace space are large. However, the representation
is well suited for the applications pursued here, specifically the robustness study
where we are interested in exact (inverted by LU) preconditioners. 

\subsection{Parameter sensitivity}\label{sec:num_robust}
We demonstrate robustness of the fractional preconditioner \eqref{eq:robust_precond} by a sensitivity
study where the physical parameters in \eqref{eq:weak} are varied such
that $10^{-9} \leq \mu_f, \kappa \leq 1$, $1\leq\lambda\leq 10^{12}$,
$10^{-2}\leq \gamma\leq 10^2$, $10^{-8}\leq \alpha \leq 1$. Since $\mu_s$ is
commonly used for rescaling we fix its value to 1. Moreover, the storage
capacity is set to 0 as this is the more challenging limit of the parameter's
range. The Biot-Stokes system is then considered on the geometry from Example
\ref{ex:naive} with the boundary configuration satisfying the assumptions
in Theorem \ref{theo:iso}. That is, the interface intersects $\Gamma^{\bu}_F$ on
the Stokes side and $\Gamma^{\bd}_P$ on the Biot side.
In turn, the fractional
operator is constructed as given in \eqref{eq:fract_laplace}. Finally, following
Example \ref{ex:naive}, the convergence criterion for the MinRes solver is a
reduction of the preconditioned residual norm by factor $10^{8}$. The action of
the preconditioner \eqref{eq:robust_precond} is then computed by LU factorization
of the $2\times 2$ velocity-displacement and the $3\times 3$ pressure blocks, respectively.

Using $\text{TH}_1$ elements, Figures \ref{fig:iters_DIR_TH_loop_alpha_C00} and
\ref{fig:iters_DIR_TH_loop_gamma_C00} present slices of the explored parameter
space. More precisely, in each subplot column-indexed by fixed value of $\mu_f$
and row-indexed by scalar permeability $\kappa$ we plot dependence
of the MinRes iterations on system size for varying Lam{\'e} parameter $\lambda$ (indicated
by color), the Biot-Willis coefficient $\alpha$ (in Figure \ref{fig:iters_DIR_TH_loop_alpha_C00}, $\gamma$ is set to 1)
and the slip-rate coefficient $\gamma$ (in Figure \ref{fig:iters_DIR_TH_loop_gamma_C00}, $\alpha$ is set to 1).
We observe that the iterations are stable for all the parameter combinations.
In fact, the iterations remain bounded between $21$-$56$ across the entire considered parameter range.

\begin{figure}[t!]
  \includegraphics[width=\textwidth]{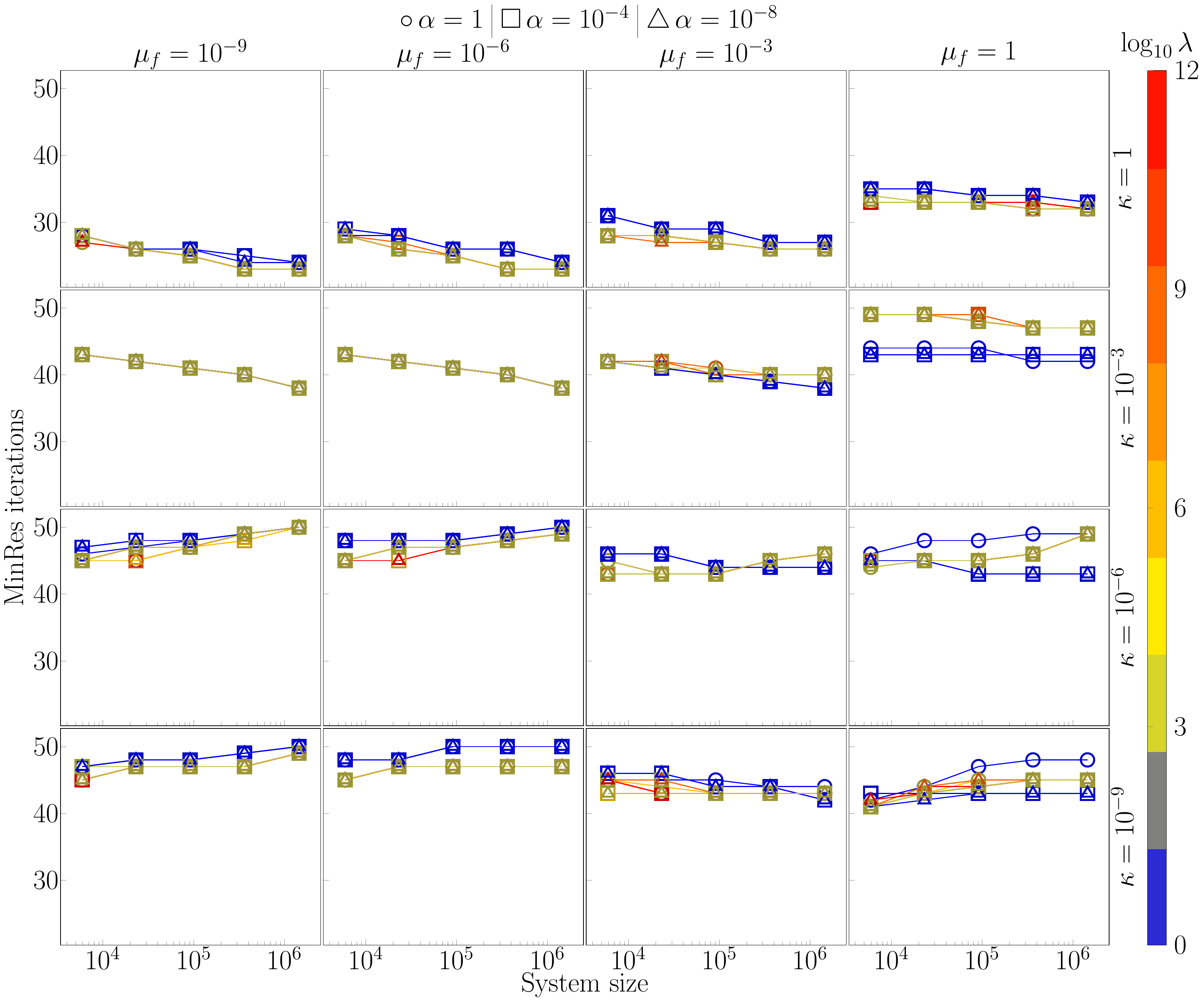}
  \vspace{-10pt}
  \caption{Performance of the Biot-Stokes preconditioner \eqref{eq:robust_precond}.
    Geometry of Example \ref{ex:naive} is used with $\Sigma$ intersecting $\Gamma^{\bu}_F$ and
    $\Gamma^{\bd}_P$. We set $\mu_s$, $\gamma$ to 1 while
    $C_0=0$. The parameters $\mu_f$, $\kappa$, $\lambda$, $\alpha$ are varied. Values
    of  the Biot-Willis coefficient are indicated by markers. In this case, the discretization uses 
    $\text{TH}_1$ elements. 
    }
  \label{fig:iters_DIR_TH_loop_alpha_C00}
\end{figure}

\begin{figure}[t!]
  \centering
  \includegraphics[width=\textwidth]{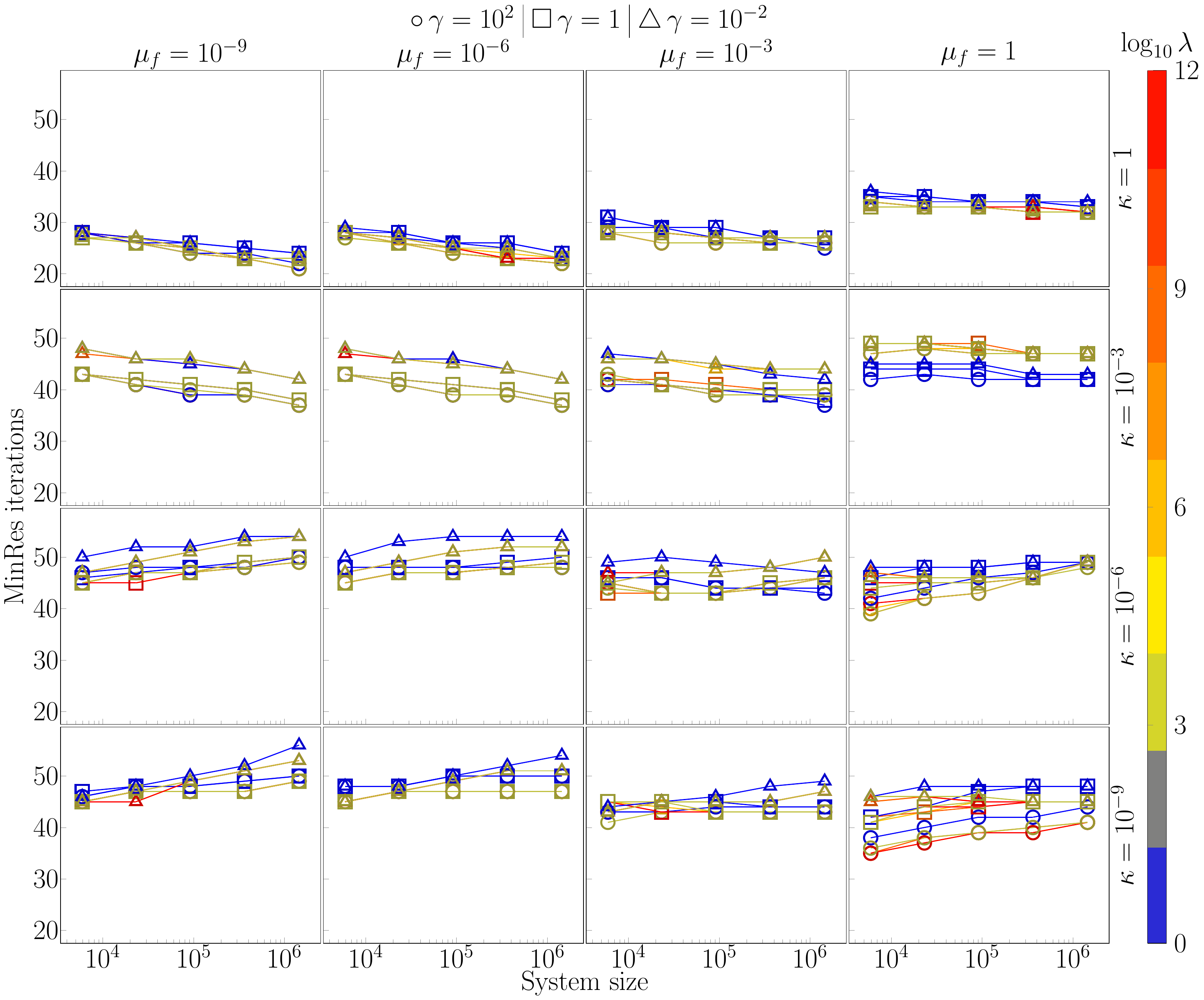}
  \vspace{-10pt}
  \caption{
    Performance of the Biot-Stokes preconditioner \eqref{eq:robust_precond}.
    The problem is setup on the geometry from Example \ref{ex:naive} with boundary conditions
    prescribed such that $\Sigma$ intersects $\Gamma^{\bu}_F$ and $\Gamma^{\bd}_P$.
    We set $\mu_s$, $\alpha$ to 1 while $C_0=0$. Parameters $\mu_f$, $\kappa$, $\lambda$, $\gamma$ are varied.
    Values of the slip-rate coefficient $\gamma$ are indicated by markers. A discretization by
    $\text{TH}_1$ is used.    
    }
  \label{fig:iters_DIR_TH_loop_gamma_C00}
\end{figure}

\subsection{Other boundary configurations}\label{sec:num_robust_neumann}
The Biot-Stokes preconditioner \eqref{eq:robust_precond} can be
extended beyond the boundary configurations assumed in the theoretical analysis,
namely, the requirement that $\Sigma$ is incident to $\Gamma^{\bu}_F$ and $\Gamma^{\bd}_P$.
We illustrate this here by letting the interface intersect the boundaries $\Gamma^{\bsigma}_F$, $\Gamma^{p_P}_P$.
Following Section \ref{sec:discrete_precond}, the configuration leads to the
fractional operator $\mu^{-1}(-\Delta_{\Sigma}+I_{\Sigma})^{-1/2}$, see \eqref{eq:eigv_neumann}.

Employing the experimental setup of Section \ref{sec:num_robust} the performance 
of the \eqref{eq:eigv_neumann}-adapted preconditioner \eqref{eq:robust_precond}
is illustrated in Figure \ref{fig:iters_TH_loop_alpha_C00} where the slice of
the parameter space for $C_0=0$, $\gamma=1$ is shown (cf. Figure \ref{fig:iters_DIR_TH_loop_alpha_C00}
where $\Gamma$ is incident to $\Gamma^{\bu}_F$ and $\Gamma^{\bd}_P$). Using
the preconditioner the iterations remain bounded between $23$ and $58$.

For CR family, the MinRes iterations are reported in Figure \ref{fig:iters_TH_loop_alpha_C0_0}.
Here, for the sake of brevity we present results only for the stronger
tangential coupling, i.e. $\gamma$ is fixed at $10^2$ and we explore robustness for
varying $\mu_f$, $\kappa$, $\lambda$ and $\alpha$. It can be seen that the
fractional preconditioner leads to bounded iterations (between $22$ and $57$ iterations are required for convergence).
We remark that the stability of CR discretization for the three-field Biot formulation is explored in \ref{sec:CR}.

\begin{figure}[t!]
  \includegraphics[width=\textwidth]{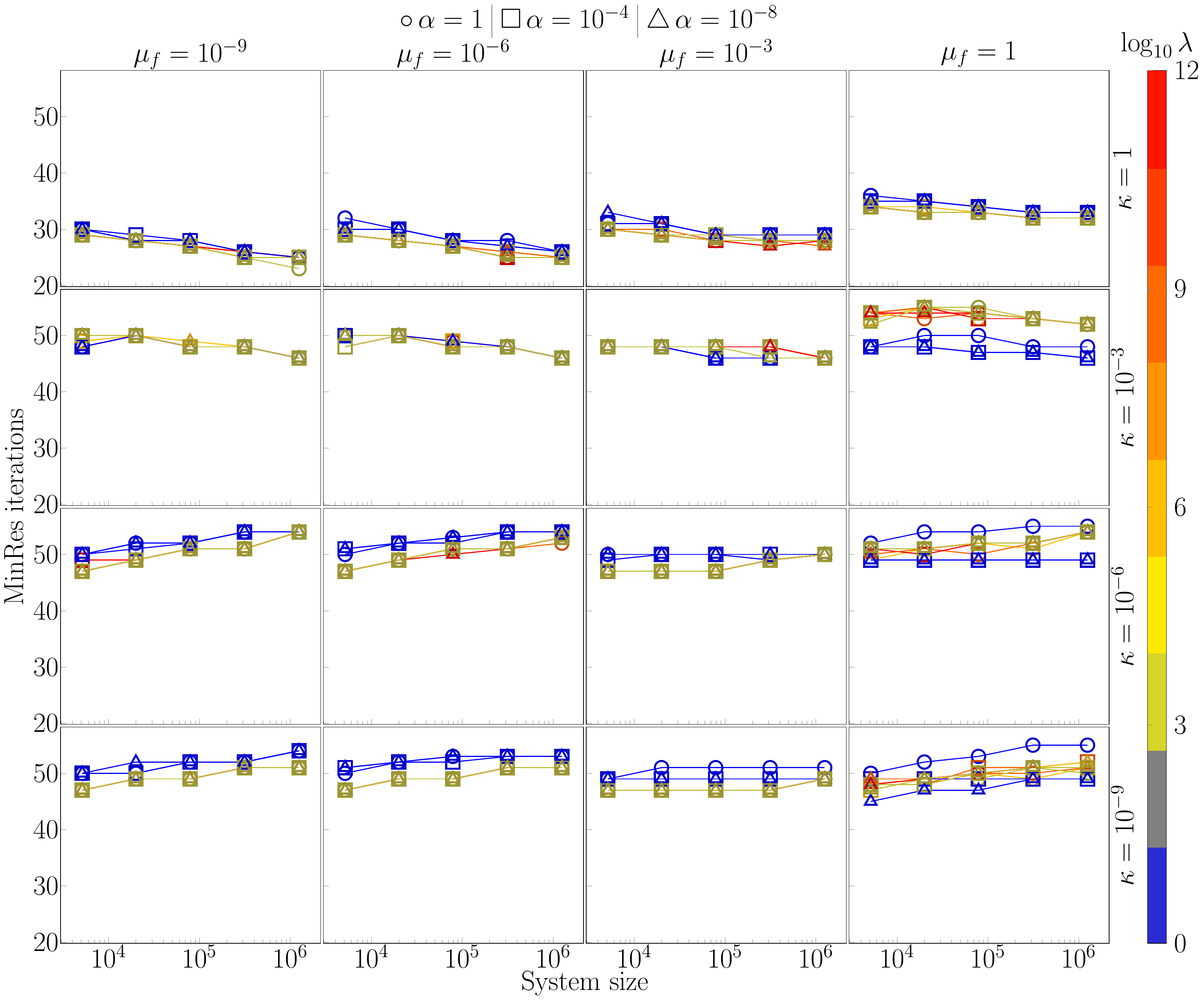}
  \vspace{-10pt}
  \caption{Performance of the Biot-Stokes preconditioner \eqref{eq:robust_precond}.
    Geometry of Example \ref{ex:naive} is used with $\Sigma$ intersecting $\Gamma^{\bsigma}_F$ and
    $\Gamma^{p_P}_P$. The fractional operator is changed to $\mu^{-1}(-\Delta_{\Sigma}+I_{\Sigma})^{-1/2}$.
    We set $\mu_s$, $\gamma$ to 1 while
    $C_0=0$. The parameters $\mu_f$, $\kappa$, $\lambda$, $\alpha$ are varied. Values
    of  the Biot-Willis coefficient are indicated by markers. In this case, the discretization uses 
    $\text{TH}_1$ elements. 
    }
  \label{fig:iters_TH_loop_alpha_C00}
\end{figure}

\subsection{Diagonal pressure preconditioner}\label{sec:diagonal}
From the point of view of computational efficiency a possible drawback\footnote{
In addition to the presence of the fractional operator on the interface.
}
of preconditioner \eqref{eq:robust_precond}, is
the fact that due to the $C$-seminorm in \eqref{eq:product-norm} the pressure block
contains a $2\times 2$ operator \eqref{eq:pressure_precond}. However, for the three-field
Biot problem, the authors in \cite{lee17} show that parameter robustness can be obtained also
if $(\varphi, p_P)$ are controlled in a simpler norm, cf. the Biot block of the operator $\cRD$ in \eqref{eq:precond_diagonal}.
Following this observation we next consider a Biot-Stokes preconditioner of the form 
\begin{equation}\label{eq:robust_precond_diag}
    \tilde{\cRF} = 
      \left(\begin{array}{cc:ccc}
\!\!\!\!\mathcal{A}_{FF} & \!\!\mathcal{A}_{FP}\!\! &  & &\\
\!\!\!\!\mathcal{A}_{PF} & \!\!\mathcal{A}_{PP}\!\! & && \\[0.5ex]
\hdashline\\[-1.5ex]
&&  \frac{1}{2\mu_f} I & & \\
&&& \!\!\!\!\!\!\bigl(\frac{1}{\lambda}+\frac{1}{2\mu_s}\bigr) I & \\
&&& & \!\!\!\!\!\!\!\bigl(C_0 + \frac{\alpha^2}{\lambda}\bigr) I - \frac{\kappa}{\mu_f} \Delta  + \frac{1}{\mu}(-\Delta_{\Sigma, 01})^{-\frac12}\!\!\!\!
\end{array}\right)^{-1},
\end{equation}
and we note that the pressure block is diagonal.

Using the computational setup of Section \ref{sec:num_robust} we compare
the preconditioner \eqref{eq:robust_precond} derived in Theorem \ref{theo:iso}
with the \cite{lee17}-inspired operator $\tilde{\cRF}$ \eqref{eq:robust_precond_diag}. For the
simple geometry of Example \ref{ex:naive} the operators are compared in
Figure \ref{fig:iters_TH_loop_alpha_C00_compare}. We observe that the
preconditioners yield practically identical number of MinRes iteration counts,
except for $\lambda=1$, $\alpha=1$ where \eqref{eq:robust_precond_diag}
requires more iterations for convergence.

\begin{figure}[t!]
  \includegraphics[width=\textwidth]{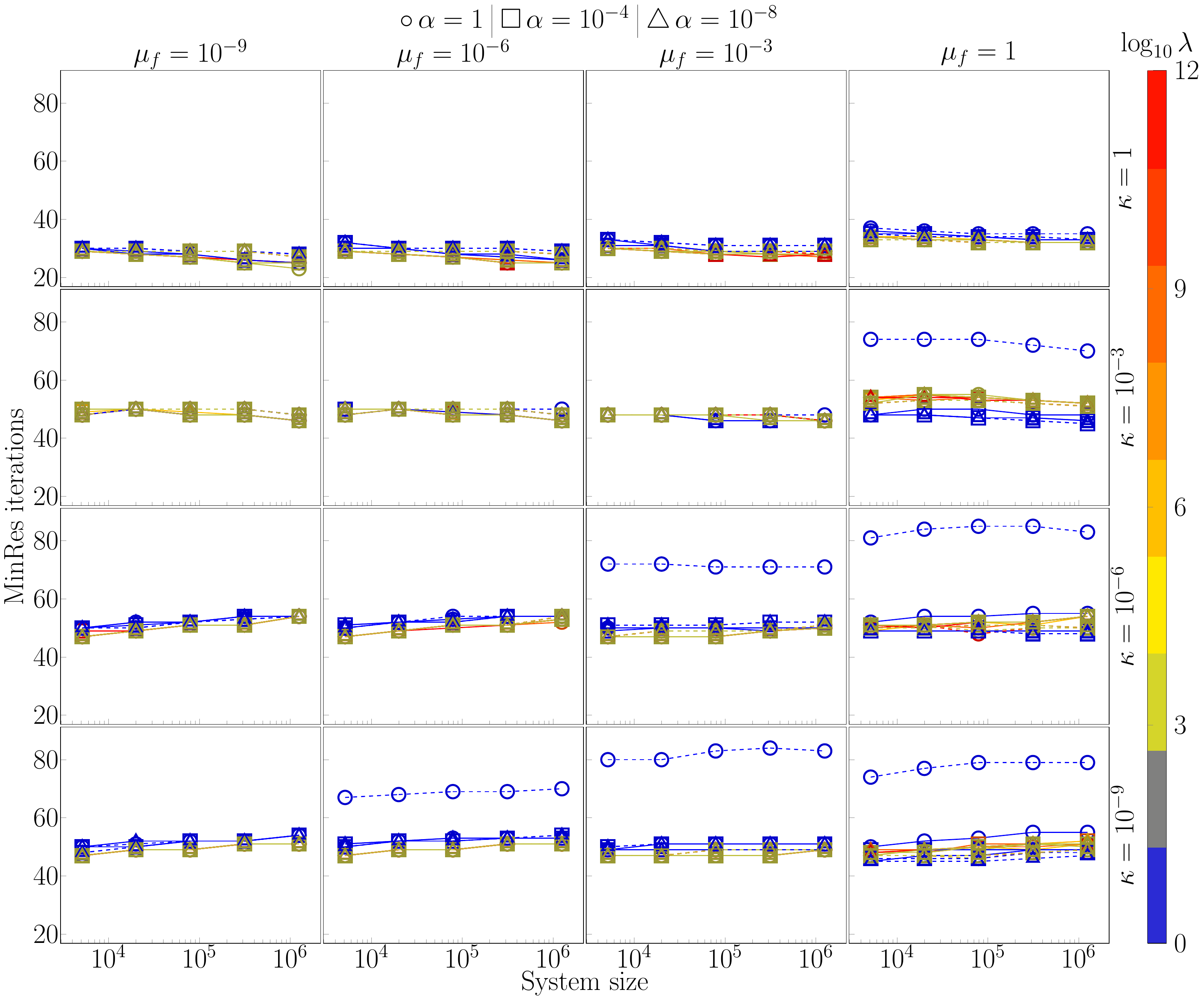}
  \vspace{-10pt}
  \caption{
    Comparison of Biot-Stokes preconditioners \eqref{eq:robust_precond} (solid lines)
    and \eqref{eq:robust_precond_diag} (dashed lines).  
    Geometry of Example \ref{ex:naive} is used with $\Sigma$ intersecting $\Gamma^{\bsigma}_F$ and
    $\Gamma^{p_P}_P$. The fractional operator in both preconditioners is changed to $\mu^{-1}(-\Delta_{\Sigma}+I_{\Sigma})^{-1/2}$.
    We set $\mu_s$, $\gamma$ to 1 while
    $C_0=0$. The parameters $\mu_f$, $\kappa$, $\lambda$, $\alpha$ are varied. Values
    of  the Biot-Willis coefficient are indicated by markers. The system is discretized 
    by $\text{TH}_1$ elements. 
    }
  \label{fig:iters_TH_loop_alpha_C00_compare}
\end{figure}

\subsection{Interfacial flow in the brain}\label{sec:num_brain}  
In the examples presented thus far the fractional preconditioners were applied
to the Biot-Stokes problem posed in geometries where the interface formed a 
simple curve (a straight segment in fact). In addition, the number of degrees of freedom associated with
$\Sigma$, and in turn the discrete fractional operators, were small.

To apply the proposed fractional preconditioners in a more practical
setting our final example concerns the interfacial flow in a brain. As
realistic brain geometries currently cannot be tackled with our
approach \eqref{eq:fract_laplace} due to the size of the interface\footnote{
  The coarsest yet still reasonably well resolved surface mesh of a $3d$ brain at
  our disposal has circa 50 thousand cells. The resulting eigenvalue
  problem is roughly 4 times larger than what can be computed on a computer with 24GB RAM.
} we choose the problem geometry as  two-dimensional slices, see Figure \ref{tab:brain}, left panels. 
We then model flow of a water-like fluid in free-flow domain that is the
space surrounding the brain known as the subarachnoid space and in the
poroelastic domain that is the brain parenchyma. We remark that the interface thus forms
a closed surface. The material parameters of the Biot model
are set following \cite{BUDDAY2015318} and we fix the slip-rate coefficient as $\gamma=1$.
The boundary of the outer spaces is assumed impermeable with $\bu=\cero$ prescribed
on most of the surface, except for regions (marked with red and orange in 
Figure \ref{tab:brain}, left) where traction is set in order to drive the flow. 

\begin{figure}[t!]
  \centering
  \renewcommand{\arraystretch}{1.3}
  \qquad\begin{minipage}{0.35\textwidth}
    \includegraphics[width=0.9\textwidth]{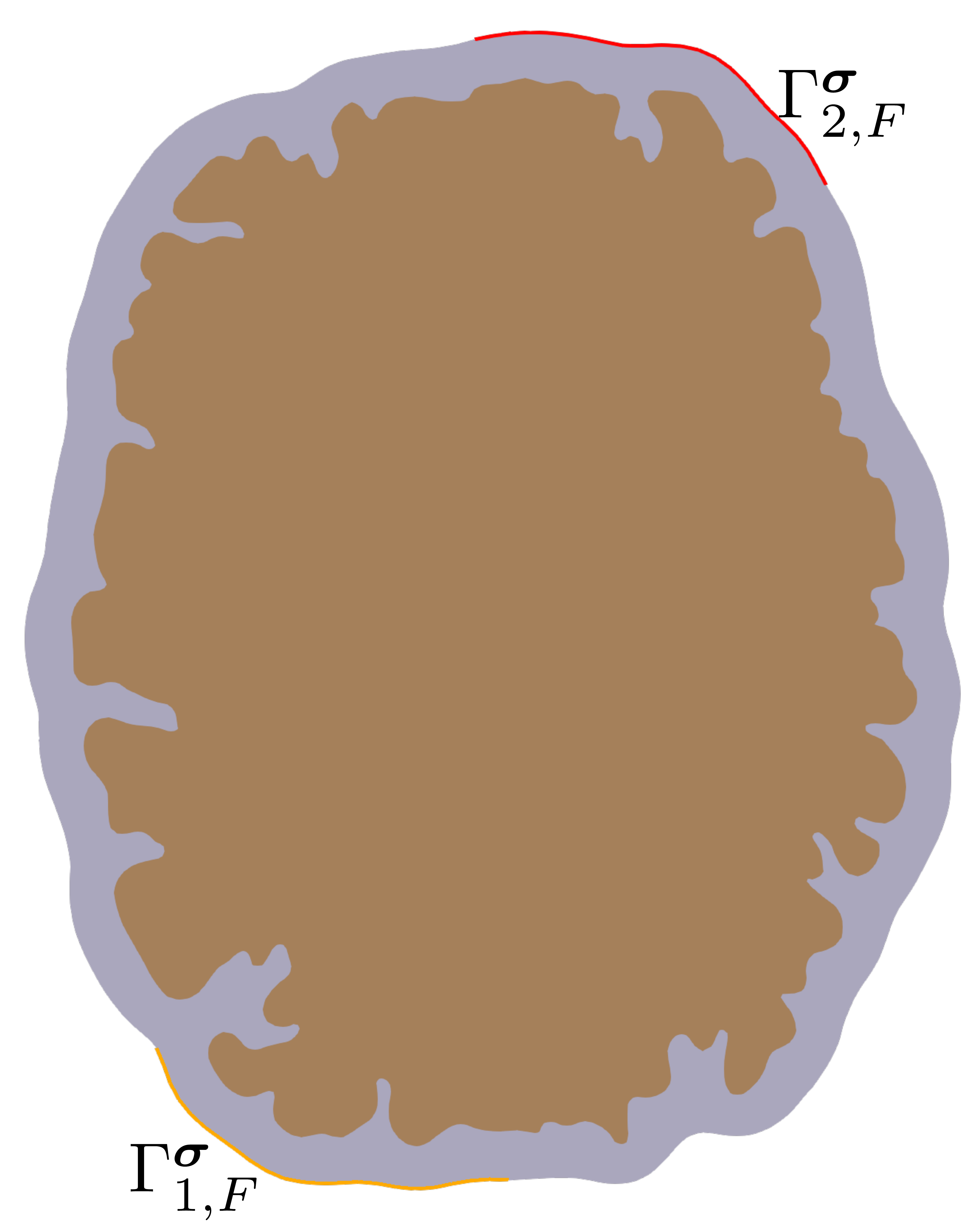}
  \end{minipage}
  \hfill\begin{minipage}{0.45\textwidth}
    \small{
    \begin{tabular}{c|ccc||cc}
      \hline
      $h\left[\text{mm}\right]$ & $\lvert \bH_{\epsilon, h}\rvert$ & $\lvert S_h \rvert$ & $\cRF$ & $\cRD$ & $\cRC$\\
      \hline
4.72     & 5114   & 134  & 65 & 369 & 82 \\
3.21     & 12919  & 266  & 77 & 445 & 98 \\
1.93     & 41545  & 530  & 75 & 473 & 104\\
1.05     & 134980 & 1066 & 81 & 513 & 114\\
0.53     & 482371 & 2118 & 83 & 525 & 119\\
0.27     & 1840764& 4234 & 82 & 521 & 119\\
      \hline
    \end{tabular}
    }
  \end{minipage}\\
  \qquad\begin{minipage}{0.35\textwidth}
    \includegraphics[width=\textwidth]{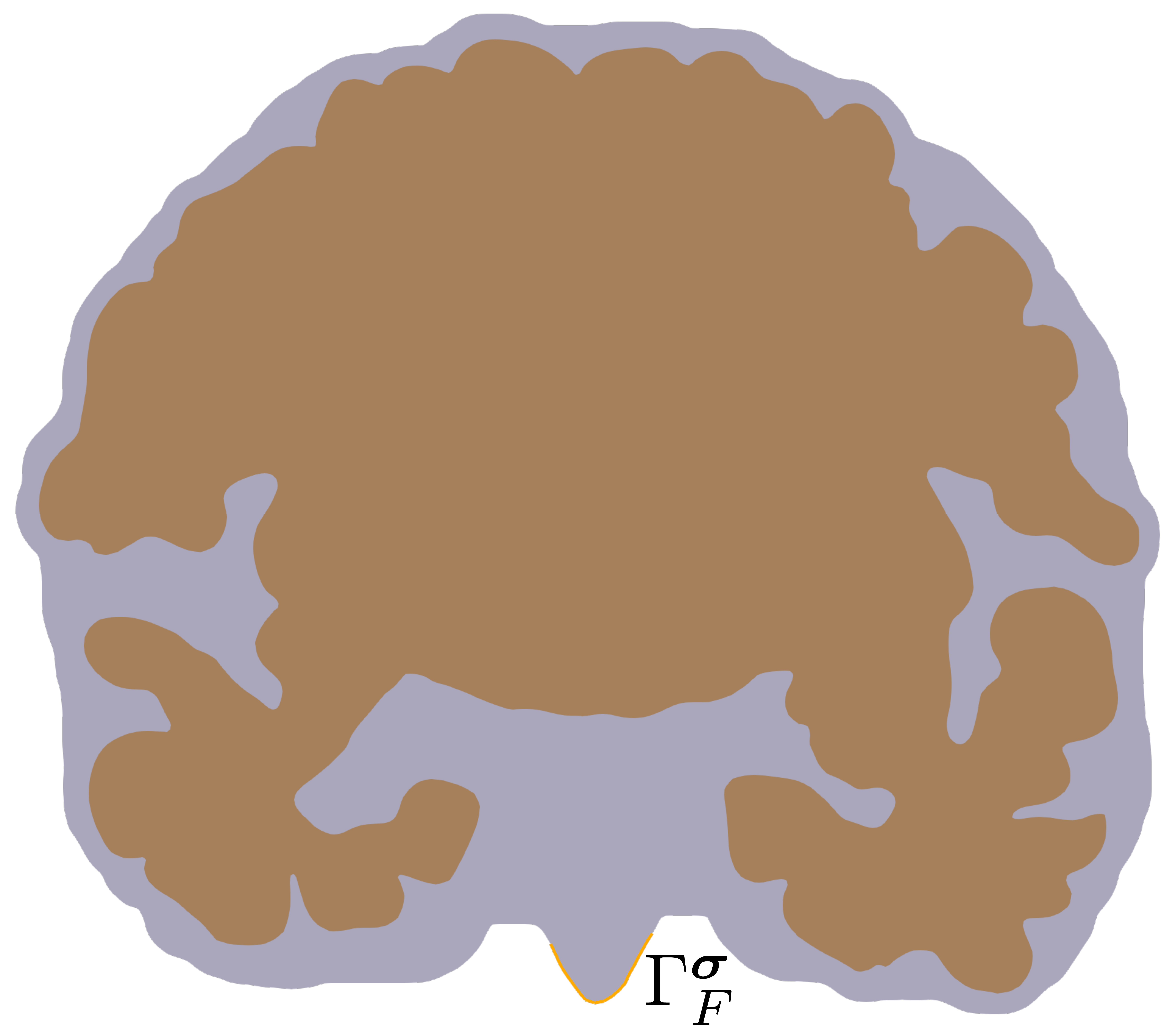}
  \end{minipage}
  \hfill\begin{minipage}{0.45\textwidth}
    \small{
    \begin{tabular}{c|ccc||cc}
      \hline
      $h\left[\text{mm}\right]$ & $\lvert \bH_{\epsilon, h}\rvert$ & $\lvert S_h \rvert$ & $\cRF$ & $\cRD$ & $\cRC$\\
      \hline
2.46 & 11681   & 178   & 81 & 413 & 101\\
1.67 & 36677   & 355   & 85 & 458 & 113\\
0.86 & 132775  & 710   & 91 & 494 & 126\\
0.44 & 496680  & 1420  & 91 & 507 & 131\\
0.21 & 1895459 & 2838  & 93 & 511 & 134\\
      \hline
    \end{tabular}
    }
  \end{minipage}  
  %
  \caption{Idealized brain sized geometry. The brain is enclosed
    in a water-filled fluid space. For horizontal slice (top) $\lvert \Sigma \rvert \approx 690\text{mm}$,
      $\lvert \Gamma^{\bu}_F \rvert \approx 391\text{mm}$,
    $\lvert \Gamma^{\bsigma}_F \rvert \approx 140\text{mm}$ while
    in coronal slice (bottom)
    $\lvert \Sigma \rvert \approx 825\text{mm}$, $\lvert \Gamma^{\bu}_F \rvert \approx 440\text{mm}$,
    $\lvert \Gamma^{\bsigma}_F \rvert \approx 24\text{mm}$. The outer boundary of the fluid
    space is assumed impermeable except for orange and red segments that form $\Gamma^{\bsigma}_F$.
    Gradually refined meshes
    of the geometry are generated with mesh size in $\text{mm}$ given in the $h$ column
    of the table. The dimensionalities of the corresponding finite element
    solution space based on $\text{TH}_1$ elements are shown together
    with the size of the trace space $S_h$, cf. Section~\ref{sec:discrete_precond}.
    The fractional preconditioner \eqref{eq:robust_precond} leads to bounded iterations, 
    and also fewer iterations are required compared to the simple preconditioners
    from Example \ref{ex:naive}. 
  }
  \label{tab:brain}
\end{figure}

Having discretized the system by $\text{TH}_1$ elements, the flow problem
is solved by a preconditioned MinRes solver starting from a 0
initial vector with relative tolerance of $10^{-8}$. As the preconditioner \eqref{eq:robust_precond} is used,
where the fractional term reads
$\mu^{-1}(-\Delta_{\Sigma}+I_{\Sigma})^{-1/2}$, cf. Section \ref{sec:num_robust_neumann}.
Note that, while $\Sigma$ is now a closed surface, the fractional operator is well
defined (and invertible) as the spectrum of the eigenvalue problem \eqref{eq:eigv_neumann} is
positive due to the $H^1$-inner product used in the definition.

The number of iterations required for convergence is tabulated in Figure \ref{tab:brain}, right. The iterations
are bounded in mesh size and (while the interface is more complex and the interface
problem much larger) are in fact comparable to those in the simple
setup of Example \ref{ex:naive} and Section \ref{sec:num_robust}.  We note
that the fractional preconditioner leads to faster convergence
of the MinRes solver especially compared to the simple
block-diagonal preconditioner $\cRD$, see Example \ref{ex:naive}.

The resulting flow and pressure fields for the two sets of simulations 
are shown in Figure \ref{fig:brain}, exhibiting localization of pressure, permeating into the porous domain following the directions dictated by the brain displacement. 

\begin{figure}[t!]
  \includegraphics[width=0.48\textwidth]{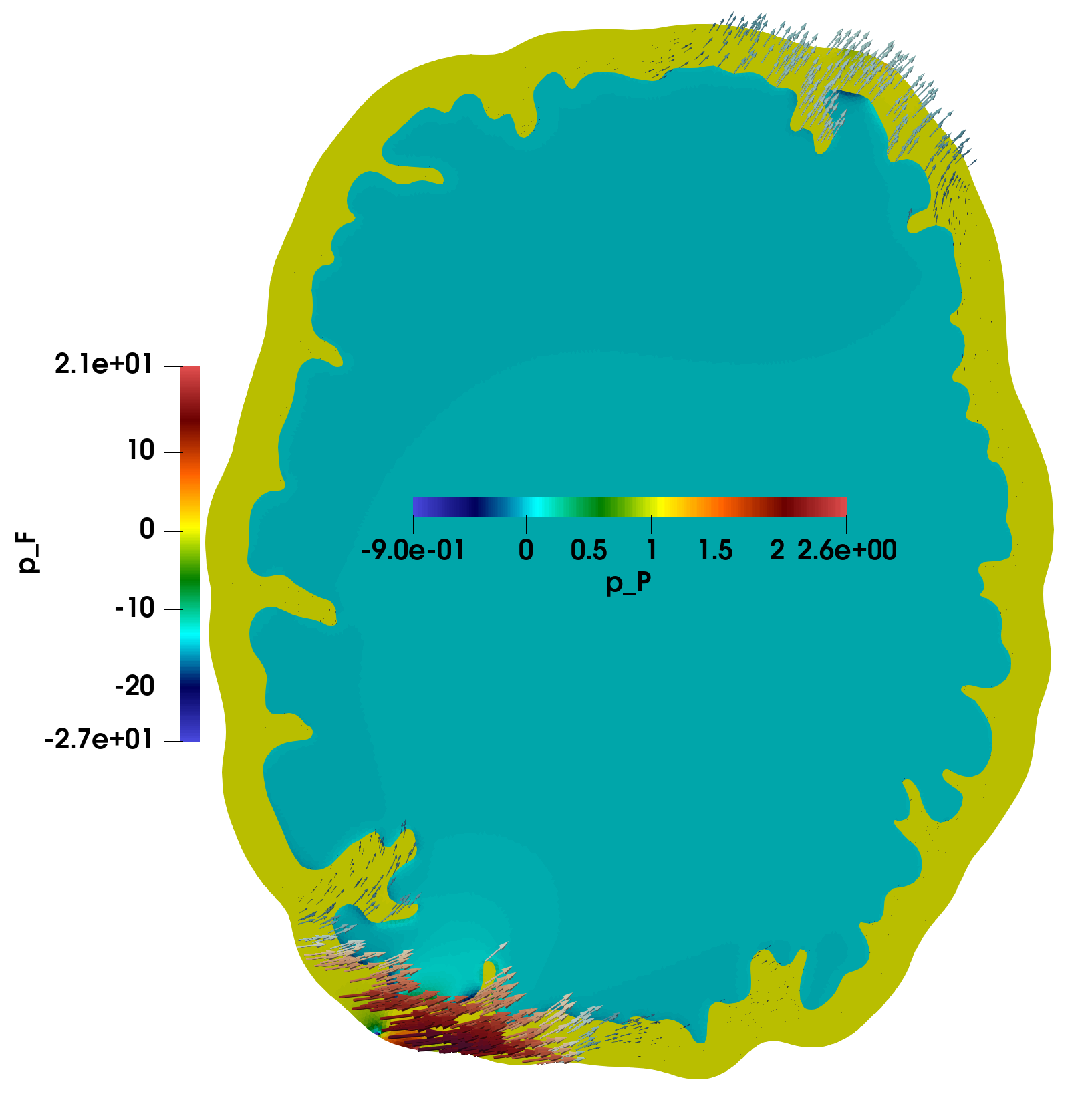}
  \includegraphics[width=0.48\textwidth]{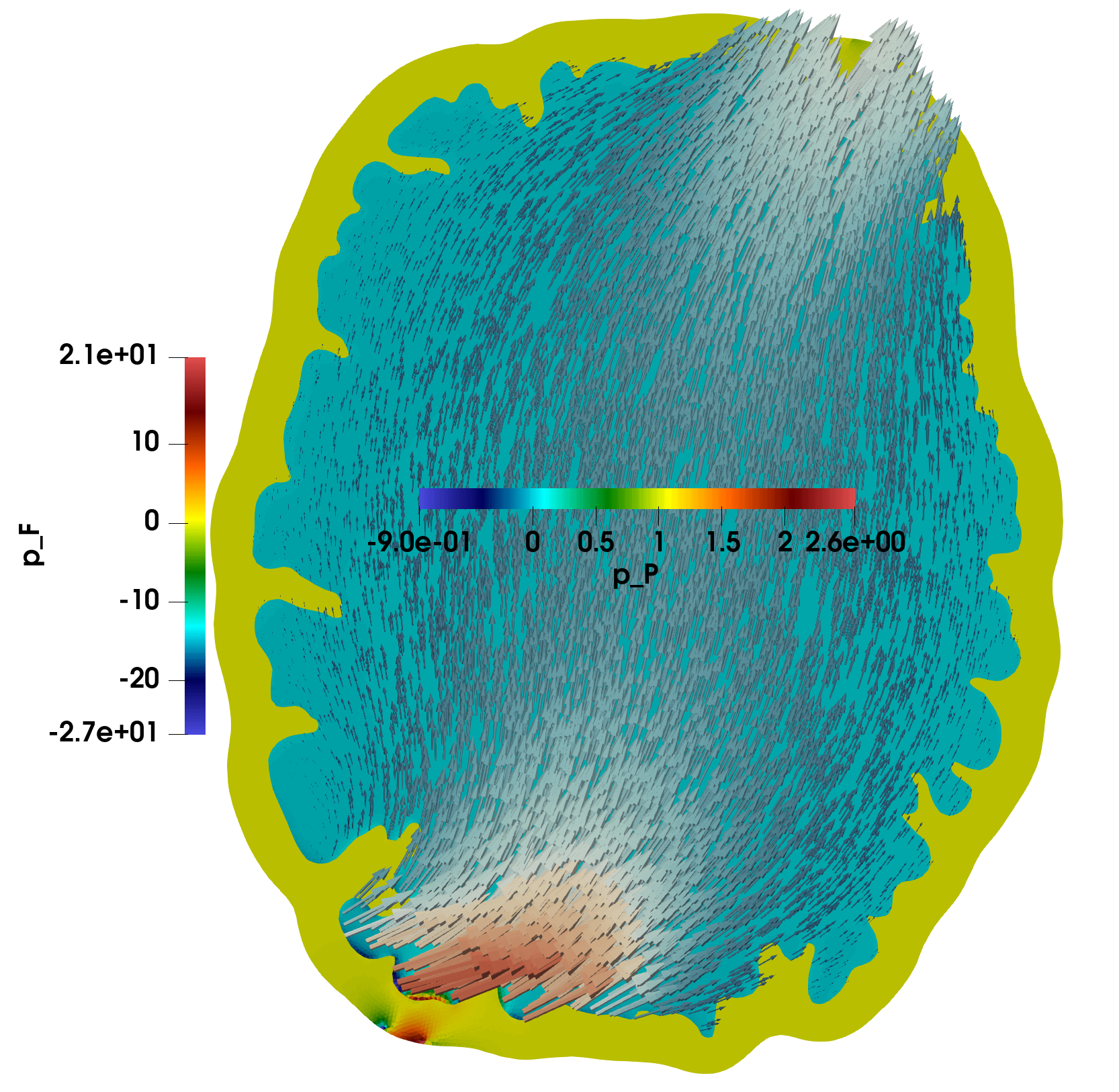}\\
  \includegraphics[width=0.48\textwidth]{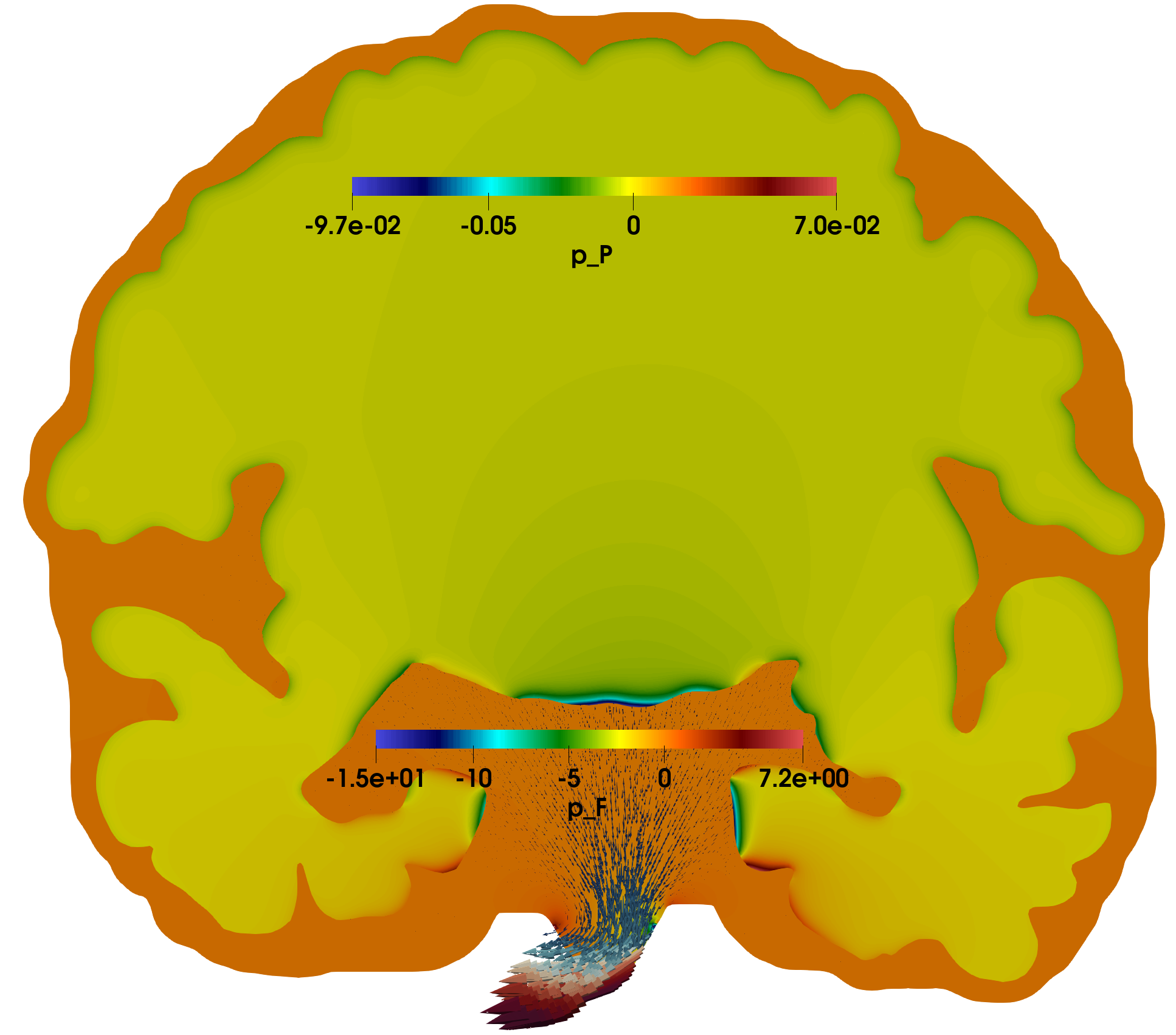}
  \includegraphics[width=0.48\textwidth]{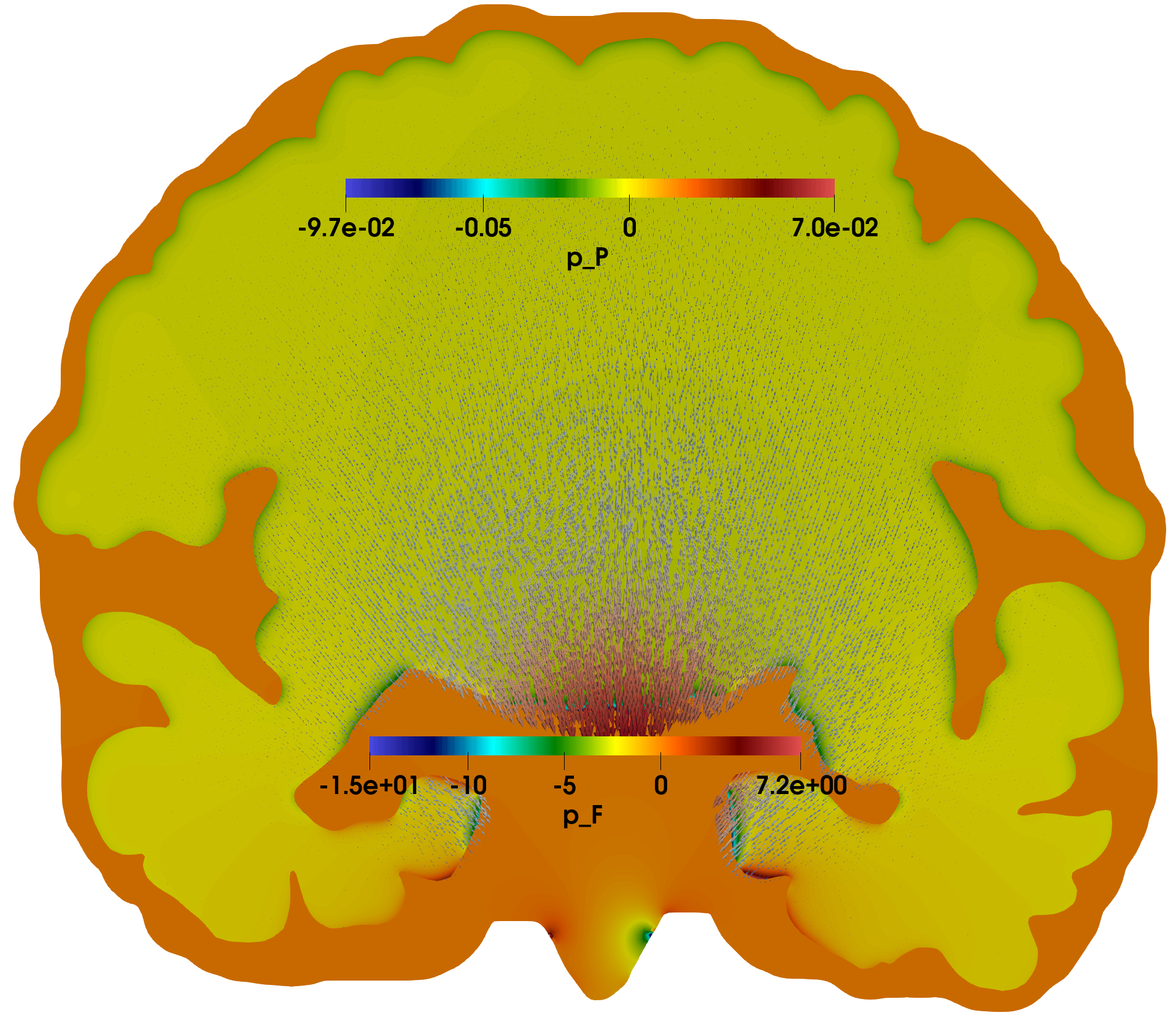}\\
  \vspace{-10pt}
  \caption{
    Interfacial flow in an idealized geometry. The traction boundary conditions in the top right and bottom left corners (horizontal
    slices), respectively at the bottom (coronal slices), induce the fluid flow (left) and the brain displacement (right).
  }
  \label{fig:brain}
\end{figure}

\bibliography{bhkmr-bib}

\bigskip 

\appendix

\section{TH$_k$ and CR finite element families}\label{section:TH-CR}
We
denote by $\{\cT_{h}\}_{h>0}$ a shape-regular family of partitions of 
$\bar\Omega$, conformed by tetrahedra (or triangles 
in 2D) $K$ of diameter $h_K$, with mesh size
$h:=\max\{h_K:\; K\in\cT_{h}\}$, and denote $\cT_{h}^F$ and $\cT_{h}^P$ the restrictions of the mesh elements to the subdomains $\Omega_F$ and $\Omega_P$, respectively. Similarly, 
by $\cE_h^F$ and  $\cE_h^P$ we will denote the restrictions to the mesh facets (edges in 2D) to the Stokes and Biot subdomains, respectively.  
We assume that the two partitions match at the interface. 
 Given an integer $k\ge1$ and a subset
$S$ of $\mathbb{R}^d$, $d=2,3$, by $\mathbb{P}_k(S)$ we will denote the 
 space of polynomial functions defined locally in $S$ and being of total degree up to $k$.
The methods that we use are based on the generalized Taylor-Hood \cite{taylor73} and Crouzeix-Raviart \cite{crouzeix73} finite element families that give 
an overall $k+1$ order of convergence 
\begin{align*}
\mathrm{TH}_k&\begin{cases} \vec{\bV}_h  =\{(\bv_h,\bw_h)\! \in\!\vec{\bV}\!:  \bv_h|_{K} \in \mathbb{P}_{k+1}(K)^d, \bw_h|_{L} \in \mathbb{P}_{k+1}(L)^d, \forall K\in \cT^F_h,L\in\cT_h^P\}, \\
	\vec{Q}_h =\{\vec{q}_{h}\in \vec{Q} \cap [C(\Omega_F)\times C(\Omega_P)\times L^2(\Omega_P)]: \\
\qquad \  q_{F,h}|_K\in\mathbb{P}_{k}(K),	\psi_h|_L\in\mathbb{P}_k(L), q_{P,h}|_L\in\mathbb{P}_{k+1}(L),	  \forall K\in\cT^F_h,L\in\cT_h^P\},\end{cases}\\
\mathrm{CR}&\begin{cases} \vec{\bV}_h  =\{(\bv_h,\bw_h)\! \in\!\bL^2(\Omega_F)\times\bL^2(\Omega_P)\!: \bv_h|_{K} \in \mathbb{P}_{1}(K)^d, \bw_h|_{L} \in \mathbb{P}_{1}(L)^d, \\
\quad \quad \int_e \jump{\bv_h\cdot\nn_e} = 0, \ \int_\ell \jump{\bw_h\cdot\nn_\ell} = 0 \ 
\forall e\subset \partial K, \ell\subset \partial L, \, K\in \cT^F_h,L\in\cT_h^P\}, \\
	\vec{Q}_h =\{\vec{q}_{h}\in \vec{Q} \cap [L^2(\Omega_F)\times L^2(\Omega_P)\times C(\Omega_P)]: q_{F,h}|_K\in\mathbb{P}_{0}(K),	  \\
\quad  \quad \psi_h|_L\in\mathbb{P}_0(L), q_{P,h}|_L\in\mathbb{P}_{1}(L),  
\forall K\in\cT^F_h\!,L\in\cT_h^P\}.\end{cases}
\end{align*}

\section{CR discretization for Biot system}\label{sec:CR}
We investigate numerically the stability of the three-field (total pressure)
formulation of Biot equations discretized by the CR family. Assuming momentarily that
$\partial\Omega_P=\Gamma^{\bd}_P$ the discrete weak problem reads:
Find $(\bd_h, p_{P,h}, \varphi_h)\in \bW_h\times Q^P_h\times Z_h$ such
that for all $(\bw_h, q_{P, h}, Z_h)\in \bW_h\times Q^P_h\times Z_h$ it holds that 
\begin{equation}\label{eq:CR-biot}
  \begin{aligned}
  \int_{\Omega_P} 2\mu_s \beps(\bd_h):\beps(\bw_h) + \int_{\cE_h^P}
  \frac{2\mu_f}{h}\jump{\bd_h\cdot\nn} \jump{\bw_h\cdot\nn}
  -\int_{\Omega_P}\varphi_h\vdiv\bw_h &= \int_{\Omega_P}\boldsymbol{f}\cdot\bw_h,
  \\
  -\int_{\Omega_P}\frac{\alpha^2}{\lambda}p_{P,h} q_{P,h} - \int_{\Omega_P}\frac{\kappa}{\mu_f}\nabla p_{P,h}\cdot\nabla q_{P,h} + \int_{\Omega_P}\frac{\alpha}{\lambda}\varphi_h q_{P,h}&= 0,\\
  %
  \int_{\Omega_P}\psi_h\left(-\vdiv\bd_h + \frac{\alpha}{\lambda}p_{P,h} -\frac{1}{\lambda}\varphi_h\right) &= 0.
\end{aligned}
\end{equation}
Note that the momentum balance equation includes a stabilization as described in \cite{burman05}.

\begin{figure}[t!]
  \includegraphics[width=\textwidth]{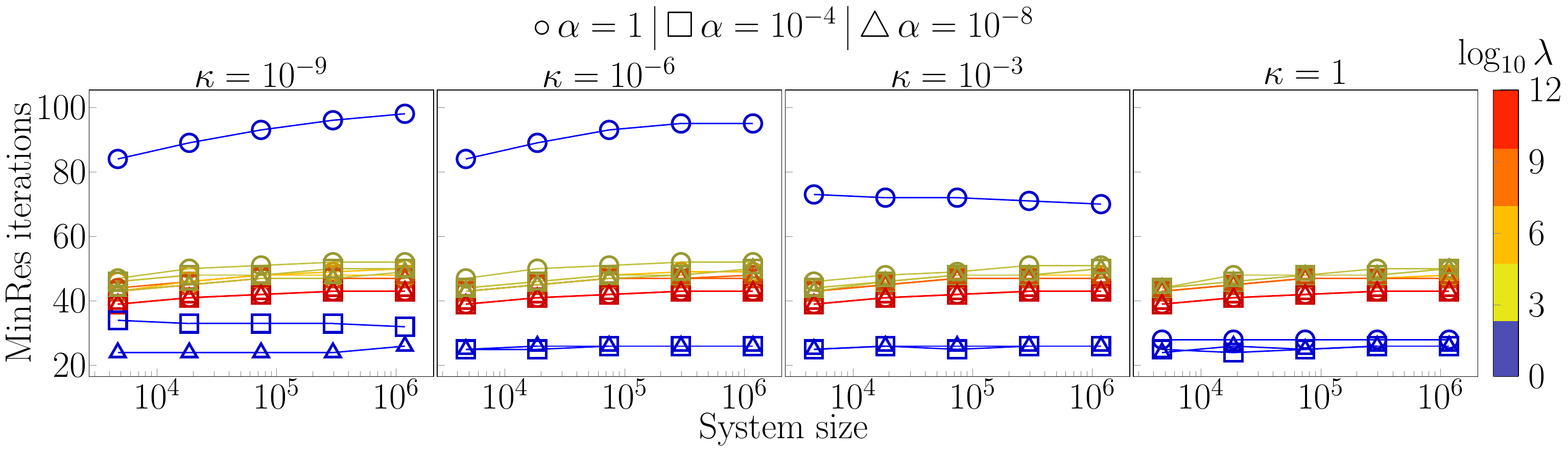}  
  \vspace{-10pt}  
  \caption{
    Total pressure Biot formulation \eqref{eq:CR-biot} with preconditioner from \cite{lee17}, and
  using a  discretization by $\text{CR}$ family.
  }
  \label{fig:biot3_CR}
\end{figure}

In Figure \ref{fig:biot3_CR} we report the number of iterations required for convergence of MinRes solver
started from random initial vector until reducing the preconditioned residual
norm by a factor of $10^{8}$. Here the preconditioner analyzed in \cite{lee17}
(that is, the preconditioner is formed by the second, fourth and final blocks of
the operator $\cRD$ in \eqref{eq:precond_diagonal}) is
used. We remark that in the numerical experiment $\Omega_P=(0, 1)^2$, and  
$\lvert \Gamma^{p_P}_P\rvert\cdot\lvert \Gamma^{\bd}_P\rvert>0$.
We observe that the iterations are bounded with respect to the mesh size as well as to
 variations in material parameters.

\begin{figure}[t!]
  \includegraphics[width=\textwidth]{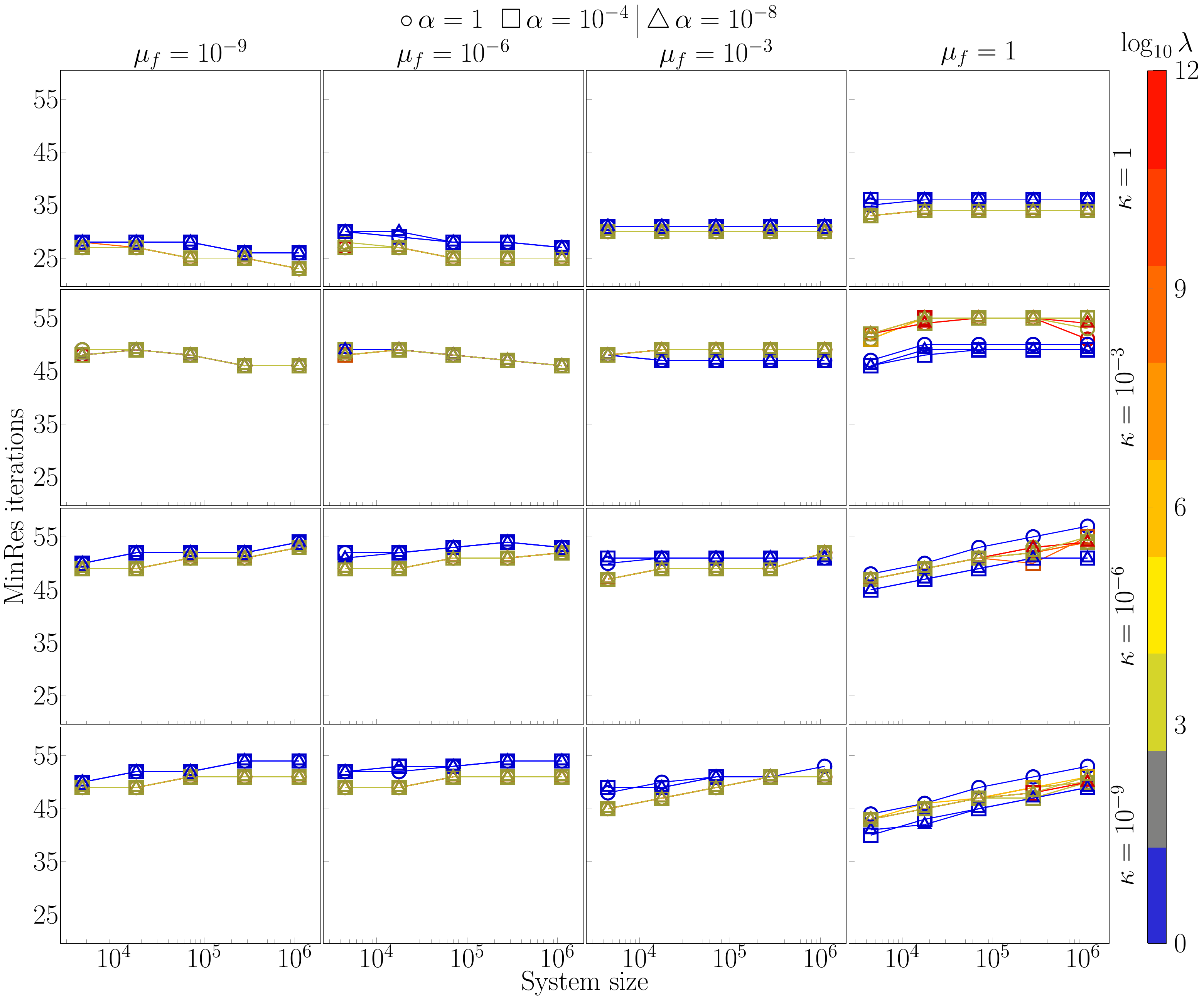}
  \vspace{-10pt}  
  \caption{
    Performance of Biot-Stokes preconditioner \eqref{eq:robust_precond}.
    Geometry of Example \ref{ex:naive} is used with $\Sigma$ intersecting $\Gamma^{\bsigma}_F$ and
    $\Gamma^{p_P}_P$. The fractional operator is changed to $\mu^{-1}(-\Delta_{\Sigma}+I_{\Sigma})^{-1/2}$.
    We set $\mu_s=1$, $C_0=0$, $\gamma=10^2$. The parameters $\mu_f$, $\kappa$, $\lambda$, $\alpha$ are varied. Values
    of  the Biot-Willis coefficient are indicated by markers.
    The discretization is done using the CR family.    
  }
  \label{fig:iters_TH_loop_alpha_C0_0}
\end{figure}


\end{document}